\documentclass[12pt]{amsart}
\usepackage{amsfonts, amssymb, amsmath, graphicx} 
\usepackage{dcolumn}
\usepackage{bm}
\usepackage{xcolor} 
\usepackage{comment} 
\usepackage[normalem]{ulem}
\usepackage{subfigure}
\definecolor{blue2}{rgb}{0,0.45,0.74}

\newtheorem{lemma}{Lemma}
\newtheorem{corollary}[lemma]{Corollary}
\newtheorem{proposition}[lemma]{Proposition}
\newtheorem{theorem}[lemma]{Theorem}

\newtheorem{remark}[lemma]{Remark}

\newcommand\TR{\mathrm{Tr}}

\newcommand\be{\begin{equation}}
\newcommand\ee{\end{equation}}
\newcommand{\tdens}{\tau}
\newcommand{\edens}{\varepsilon}

\newcommand{\ER}{Erd\H{o}s-R\'enyi}
\newcommand{\G}{\mathbb{G}}
\newcommand{\B}{\mathbb{B}}

\newcommand{\R}{\mathbb{R}}
\newcommand{\tet}{\textit}
\newcommand{\Value}{worth}
\newcommand{\Values}{worths}
\newcommand{\regular}{\mathcal{W}}
\newcommand{\reduced}{\widetilde {\mathcal{W}}}

\title{Emergence in Graphs with near-extreme constraints}
\author{Charles Radin and Lorenzo Sadun}
\address{Department of Mathematics, University of Texas at Austin}
\email{radin@math.utexas.edu, sadun@math.utexas.edu}

\begin{document}

\begin{abstract} 
We consider entropy-optimal graphons associated with extreme and near-extreme constraints on the densities of edges and triangles. We prove that the optimizers for near-extreme constraints are unique and multipodal and are perturbations of the previously known unique optimzers for extreme constraints. This proves the existence of infinitely many phases. We determine the podal structures in these phases and prove the existence of phase transitions between them.
 \end{abstract} 
\maketitle

\section{Setting and results}

In this paper we determine the structure of typical
large graphs with a specified density $e$ of edges and 
a near-extreme density $t$ of triangles. It is known that
the structure of all but an exponentially small 
fraction of such graphs are described by a graphon that
maximizes a certain entropy functional subject to 
certain integral constraints, assuming that the entropy
maximizer is unique up to equivalence. (See Section \ref{Definitions} for relevant background.) We show that
this graphon is indeed unique and takes a simple multipodal form, with parameters that are piecewise analytic functions of 
$(e,t)$. We thereby prove the existence of infinitely many phases and infinitely many phase 
transitions. These results are stated in Theorems \ref{thm:flat}, 
\ref{thm:scallop}, \ref{thm:orderparameter} and \ref{thm:top}.

Our solutions to the entropy maximization problem rely
on the theory of Lagrange multipliers and on a new
quantity, called ``\Value'', that applies to 
columns of graphons. In Section \ref{sec:Lagrange} we 
develop the theory of Lagrange multipliers for 
graphons, define \Value, and prove in Theorem 
\ref{thm:worth} that all of the columns of 
an entropy-maximizing graphon must maximize \Value.

\subsection{\label{Background} {The background}} 
The boundary curves in Figure \ref{fig:phase_space} show the extreme accessible values of pairs $(\edens,\tdens)$ of densities of edge and triangle subgraphs in asymptotically large simple graphs \cite{Raz}. Values throughout the interior are also accessible, and can be studied, using the formalism of graph limits or graphons developed in 2006-2011 by Lov\'{a}sz et al \cite{BCL,BCLSV,LS1,LS2,LS3} (see \cite{Au, DJ} and \cite{Lov} for background), and the Large Deviation Principle (LDP) of $\G(n,p)$ graphs \cite{CV}.

\begin{figure}[ht]
\includegraphics[width=4in]{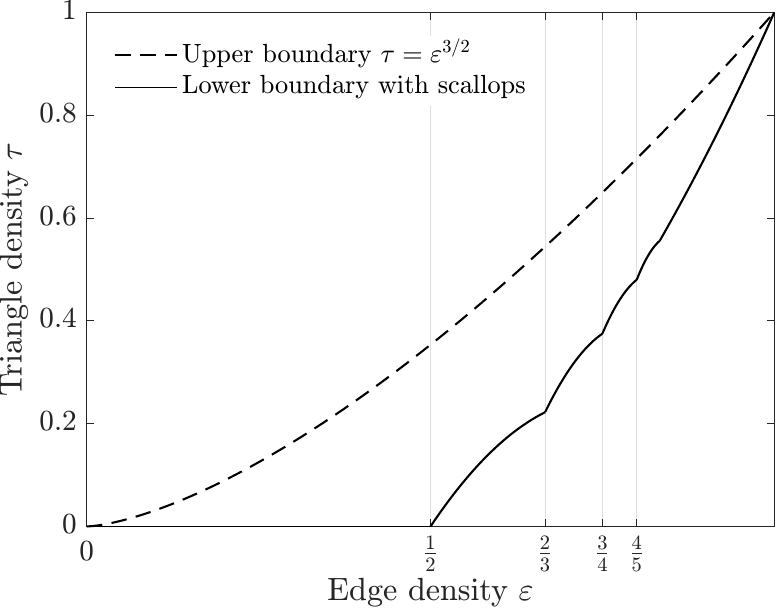}
\caption{The Razborov triangle, made from curves displaying  the extreme values of pairs of accessible edge and triangle densities. The curvature of the ``scallops'' on the lower 
right is exaggerated for visibility.}
\label{fig:phase_space}
\end{figure}

The  structure of the graphs on the
boundary of Figure \ref{fig:phase_space} appeared in 2012  in a preprint of \cite{PR}.
Following as it did soon after the development of graphons and the LDP for $\G(n,p)$, this led to a series of papers aimed at analyzing emergent (large scale) structures when the edge/triangle constraints $(\edens,\tdens)$  were near the ``\ER{}'' curve $\tdens=\edens^3$, associated with graphs in $\G(n,p)$. 
These papers used the LDP to introduce a {\em Boltzmann entropy}, $\B(e, t)$, which measures the exponential rate of growth of the
number of large graphs with edge and triangle densities $(\edens, \tdens) \approx (e, t)$, a notion copied from statistical mechanics. See \cite{RS2} for a precise statement and proof of this theorem, and the important fact that $\B(e,t)$ is not convex in $(e,t)$.

Computer simulations led to the conjecture \cite{KRRS1} of many distinct emergent ``phases'' throughout the interior of Figure \ref{fig:phase_space}, as seen in Figure \ref{fig:conjecture}. Eventually emergent phases were determined for parts of the three regions $F(1,1)$ \cite{KRRS2}, $B(1,1)$ \cite{NRS3} and $A(2,0)$ \cite{NRS4} in Figure \ref{fig:conjecture} near the \ER{} curve.

\begin{figure}[ht]
\includegraphics[width=4in]{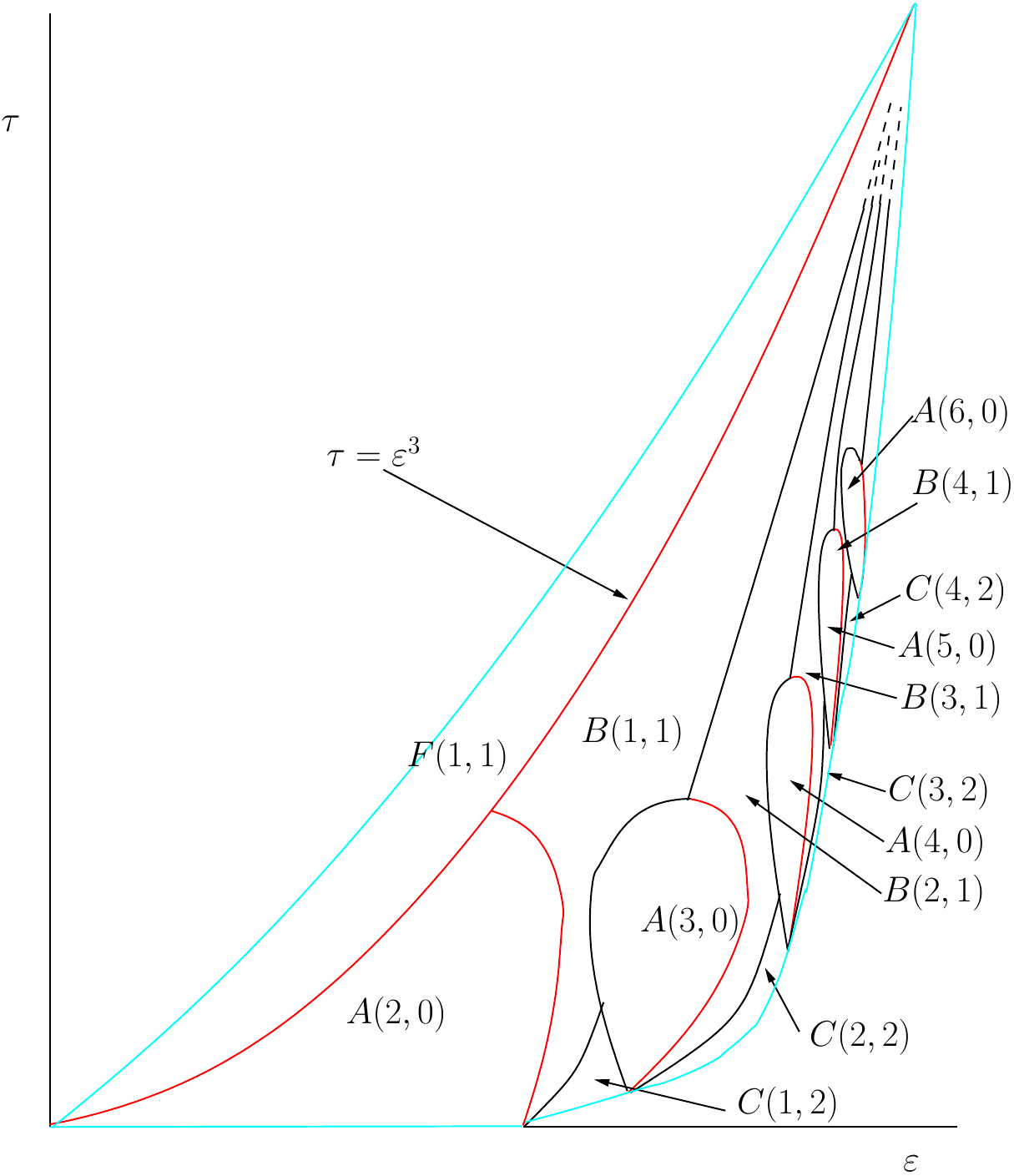}
\caption{Schematic drawing of a conjecture from 2017 \cite{KRRS1}, 
based on computer simulations of entropy-optimal graphons associated with the phases of large 
graphs with edge and triangle constraints.}
\label{fig:conjecture}
\end{figure} 
{

\subsection {\label{Connections} {Connections with other research}}

There is widespread interest in analyzing network and graph data \cite{New}, and statistical methods have developed to support this, including a family of Exponential Random Graph Models (ERGMs). In \cite{CD} a number of practical problems with ERGMs were isolated and treated using the large deviations principal of \ER{} graphs. In \cite{RS1} a Boltzmann entropy $\B(\bar \tau)$ was introduced which, together with the LDP, allowed the analysis of `exponentially most' large finite graphs with constraints on the densities $\bar \tau$ of some subgraphs; see \cite{RS2} for a review of these results in a general setting. Using $\B(\bar \tau)$ some of the weaknesses of ERGMs discussed in \cite{CD} can be quantified. This is discussed in detail in Section \ref{sec:No-ERGM}.

Our paper is concerned with the emergence \cite{Lau} of large scale structure in graphs, for instance of podes. By far the most highly developed mathematical formalism concerned with emergence is equilibrium statistical mechanics. 
In our work we have sought to create an analogous formalism for simple
constrained graphs, motivated by the Razborov triangle, Figure \ref{fig:phase_space}. 

One of our goals is to better understand constrained graphs, to extend the classic study of extremal graphs, of which \cite{PR} is a distinguished result \cite{AlK, Bollobas-Book04}. 
Another is to provide an example for other optimization problems facing similar obstacles. For this reason we give here a brief sketch of emergence in statistical mechanics \cite{Ru, Yeo, Tou}, emphasizing the significance of {\em entropy, free energies} and especially {\em convexity}.

A common structure underlying the mathematics of edge-triangle graphs and the mathematics of statistical mechanics is constrained optimization.
There are several ways to view equilibrium statistical mechanics as constrained
optimization on a space of many-particle configurations \cite{Isr,El,Ru2,Lan}, all of which involve the global optimization of any of a range of free energy functionals, or the entropy. The entropy 
in statistical mechanics is a measure of the number of possible 
particle configurations with given constraints. It is a fundamental quantity. It is no exageration to view statistical mechanics as built on the {\em convexity} of this 
entropy (see the lectures of Lanford in \cite{Lan} and the introduction by Wightman in \cite{Isr}).

The convexity of the entropy allows one to analyze the system without loss of information by the use of a variety of free energies \cite{Tou}, such as
the Gibbs free energy $G(p,T)$, which is more familiar than the entropy. 

\begin{figure}[ht]
\includegraphics[width=4in]{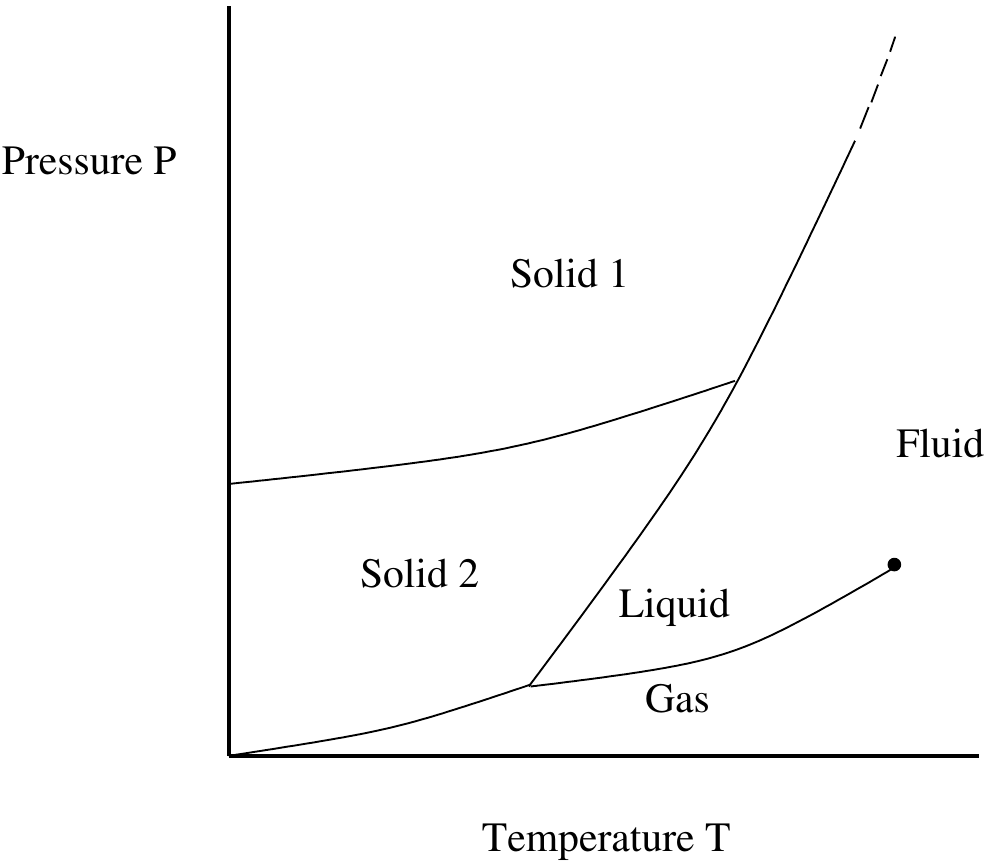}
\caption{This is a crude sketch of the phases of bulk matter, separated by transition curves. 
There are more than 20 known different solid phases of water, different crystalline structures.}
\label{fig:thermo}
\end{figure}

Figure \ref{fig:thermo} shows a primitive thermodynamic phase diagram, illustrating the pattern of solid and fluid phases in a physical bulk material (the phases represent the large particle-number scale in the emergent picture), as functions of constraint parameters pressure $p$ and temperature $T$ \cite{Pi}.
The experimentally measurable Gibbs free energy $G(p,T)$ \cite{Yeo, Sal}
is found to vary smoothly within each region, but shows 
singular behavior when constraints cross some lower dimension curves (see \cite{Ru2} and section VI in \cite{Isr}), where bulk material properties such as mass density and heat capacity can change abruptly. 

In the analogy between constrained graphs and 
statistical mechanics, discussed for instance in \cite{PN2}, \cite{St} and \cite{CD}, a large simple graph $G$ contains many edges, which play the role of individual particles, and the number of copies of some subgraph $H$ in $G$, such as 
triangles, plays the role of the total (potential) energy of $G$. A key to 
understanding emergent phases in both statistical mechanics and constrained graphs is defining an {\em order parameter} 
\cite{An}, which is a function that is identically zero in one phase but
nonzero in another. We  show its use in Section \ref{OrderParam}.

Unfortunately, with graphs there is no analog of equivalent free energies. (This is analyzed in Section \ref{sec:No-ERGM}.)
This `inequivalence of ensembles' presents a serious obstacle. In statistical mechanics free energies provide considerable technical advantages. In their absence we had to develop replacement tools.
One such tool, which we introduce in this paper, is the
``\Value'' $W(C)$ that we associate to each column $C$
of a graphon. 

We note that some important optimization problems also suffer from 
nonconvexity complications. For instance, the form of optimal transport 
theory developed by Kantorovich and Brenier was based on convex analysis.
It required significant developments over several years to allow the original {\em nonconvex} Monge 
problem to access that convex analysis. (For an introduction to optimal transport see the preface in 
\cite{San} or chapter 1 in \cite{Mag}).

The emergence of structure in large but finite physical materials is
dramatic, {\em particularly the diversity of solid phases}, such as the graphite and diamond phases of carbon. (There are more than 20 distinct solid phases of water.) All the richness displayed by the phases of all materials, not 
just the pure elements but also the huge number of compounds 
like water and alcohol, is created by the electromagnetic interaction within 
and between molecules. All of inorganic 
chemistry comes down to the 100 or so different (integer) electric 
charges of atomic nuclei. 

Emergence gives rise to the complicated 
structure of ``water'' in Figure \ref{fig:thermo} and in the edge-triangle model 
in Figure \ref{fig:conjecture}. The study of the emergence of such diversity from the interaction of 
invisibly small components of a small number of types has led to a great 
deal of interesting mathematics in the past twenty years.
It is this promising history which was the motivation to bring the richness of emergent structure in
statistical mechanics, which is built on the {\em convexity} of its 
entropy,
into the \tet{nonconvex} setting of the Boltzmann entropy $\B$ of constrained graphs. 

\bigskip
\subsection {\label{Definitions} {Definitions}}

We include here the definitions and concepts
needed to state our results precisely and to prepare the
reader to follow the proofs. 


\subsubsection{Graphons and the cut topology}
\phantom{asdf}

Intuitively, graphons are limits, as node number $n\to 
\infty$, of the adjacency matrices of simple graphs, 
with the nodes mapped to the interval $[0,1]$. For 
background up to its publication in 2011 and an 
encylopedic treatment of various aspects of graphons we 
recommend the book \cite{Lov} by Lov\'{a}sz. For a 
more recent treatment, see \cite{Ch}.

\noindent \textbf{Definition of graphons.} The space $\regular$ of graphons is the quotient, of the space of Borel measurable functions
$g: [0,1]^2 \to [0,1]$ satisfying $g(x,y)=g(y,x)$, by the identification of functions that differ only on a set of Lebesgue measure zero.

The cut topology on $\regular$ is metrizable with the following {metric}.

\noindent \textbf{Definition of the cut metric.} 
\be d_{cut}(g_1, g_2) \equiv \sup_{S,T}\Big|\int_{S \times T} 
\big [ g_1(x,y) - g_2(x,y) \big ] \,dx\, dy\Big|, 
\ee
where $S,T$ are measurable subsets 
of $[0,1]$.
%
%
%
%

Note that the absolute value goes outside the
integral!

\subsubsection{Subgraph densities and weak equivalence}\phantom{asdf}

For a simple graph  $F=(V,E)$ on $k$ nodes, we
define the (homomorphism) density of $F$ in the graphon $g$ as
\be \label{eq:density} \tdens_F(g) = \int_{[0,1]^k} \prod_{{i,j}\in E} g(x_i,x_j) dx_1\,dx_2\,\ldots,dx_k.
\ee
It can be proven \cite{Ch} that $\tdens(F)$ is continuous on $\regular$ in its cut topology.
We will concentrate on the densities 
\be \edens(g) = \iint g(x,y) \, dx \, dy \ee
and
\be \tdens(g)  = \iiint g(x,y) g(y,z) g(z,x) \, dx \, dy \, dz\ee
of edges and triangles, respectively. 

The space $\regular$ of ordinary graphons is not
compact in  the cut topology. To obtain a compact space $\reduced$ we take the
quotient of $\regular$ by a {\bf weak equivalence}
relation associated with subgraph densities.

\noindent \textbf{Definitions of weak equivalence and reduced graphons.} 
Graphons $g_1$ and $g_2$ are {weakly equivalent} 
if $\tdens_F(g_1)=\tdens_F(g_2)$ for every simple subgraph $F$. Elements of the quotient space $\reduced$ are called reduced graphons.

For our purposes, it is useful to use a different description of weak equivalence. 
The group of measure-preserving transformations of 
$[0,1]$ acts naturally on $\regular$. If $g$ is a graphon and $\sigma$ is such a 
transformation, we define, in terms of any representative modulo measure zero:
\be g^\sigma(x,y) = g(\sigma(x), \sigma(y)). \ee
We say that $g$ and $g^\sigma$ are {\bf group equivalent}. Note that $\tdens_F(g^\sigma)= \tdens_F(g)$, 
thanks to a simple change-of-variables in the integral (\ref{eq:density}), so group equivalence implies weak equivalence. 

We next define a pseudometric $\delta_{cut}$ on $\regular$ that measures how far two graphons are from 
being group equivalent.

\noindent \textbf{Definition of $\delta_{cut}$.}
\be 
\delta_{cut}(g_1,g_2)=\inf_{\sigma_1,\sigma_2} d_{cut}(g_1^{\sigma_1},g_2^{\sigma_2}). 
\ee

It can be proven \cite{Ch} that $\delta_{cut}(g_1,g_2)=0$ if and only if $g_1$ and $g_2$ are weakly equivalent.  
The pseudometric $\delta_{cut}$ on $\regular$ then descends 
to a metric $\delta_{cut}$ on 
$\reduced$: $\delta_{cut}([g_1],[g_2]) = \delta_{cut}(g_1,g_2)$, where $[g]$ denotes the weak 
equivalence class of $g\in \regular$.

\subsubsection{Compactness, the LDP of $\G(n,p)$ graphs and the Shannon entropy}\phantom{asdf}

A major result in the graphon formalism is the

\smallskip

\noindent \textbf{Compactness theorem.}
$\reduced$ is compact in the topology defined by the metric $\delta_{cut}$.

Our results rely heavily on the LDP of $\G(n,p)$ 
graphs, which is expressed in terms of graphons. 
For an elegant and concise reference for both the 
graphon formalism and the LDP we recommend the book \cite{Ch} of Chatterjee.\footnote{Note, however, that
Chatterjee works with a slightly different metric on
$\regular$, defining 
\[ d'_{cut}(g_1, g_2) = \sup_{a,b}\Big|\int_{[0,1]^2} 
a(x)b(y)\big [ g_1(x,y) - g_2(x,y) \big ] \,dx\, dy\Big|,\]
where $a$ and $b$ are Borel measurable maps from $[0,1]$ to $[-1,1]$. $d'_{cut}(g_1, g_2)$ is bounded 
above and below by multiples of $d_{cut}(g_1, g_2)$, 
so Chatterjee's topological results based on $d'_{cut}$ and the 
corresponding $\delta'_{cut}$ apply equally well to 
$d_{cut}$ and $\delta_{cut}$.}

We next introduce some terms associated with the LDP.

The {\bf Shannon entropy} of a graphon $g$ is  
\be S(g) = \iint H \big (g(x,y)\big ) \, dx \, dy, \ee
where
\be \label{entropy function} H(u) = - [u \ln(u) + (1-u) \ln(1-u)]
\ee
is the usual entropy of independent coin flips with probability $u$ of getting heads.
Note that $H(u)$ is concave down and that $H'(u) = \ln \left (\frac{1-u}{u} \right )$ diverges as 
$u$ approaches 0 or 1. 
$S(g)$ is invariant under measure-preserving
transformations of $[0,1]$ and defines a function 
(also denoted $S$) on $\reduced$. $S$ is also minus the rate function of the LDP of $\G(n,p)$ graphs. The LDP relates the number of 
large graphs associated with an (open or closed)
subset of $\reduced$ with the supremum of $S$ on 
that subset; see \cite{Ch} for more details on the LDP.

Let $\regular_{e,t}$ be the set of graphons with 
$\edens(g)=e$ and $\tdens(g)=t$ and let $\reduced_{e,t}$ be the corresponding set of reduced
graphons. 

The {\bf Boltzmann entropy} function $\B(e,t)$ can be 
understood in two ways. One is as the exponential rate 
of growth, as the number of nodes $n$ diverges, of the 
number of graphs on $n$ nodes with edge/triangle 
densities $(e,t)$. For this paper, it is useful to use 
the fact, proven in \cite{RS1,RS3}, that $\B(e,t)$ equals
the maximum of $S([g])$ on $\reduced_{e,t}$, which is 
the same as the maximum of $S(g)$ on $\regular_{e,t}$. 

Most of this paper is devoted to maximizing
$S$ on $\regular_{e,t}$ or $\reduced_{e,t}$. That is, we focus on the {\em constrained optimization} of $S$.
The optimizing graphons are called {\bf entropy-optimal graphons} 
(or ``{\bf optimal graphons}'', for short) for the given constraints $(e,t)$. When we refer to an optimal graphon being
unique we always mean that the optimal {\em reduced} graphon is unique. 

The conjecture \cite{RS3, KRRS1} that constrained 
optimization of $S$ would give rise to the rich 
phenomena of Figure \ref{fig:conjecture}, much of which 
is finally proven in this paper, was by analogy of 
$\B(e,t)$ with the Boltzmann entropy in statistical 
mechanics, as discussed in subsection 
\ref{Connections}.

\subsubsection
{Multipodality, phases and phase transitions}\phantom{asdf}

A graphon is said to be {\bf k-podal} if we can partition
$[0,1]$ into $k$ measurable sets $I_1, \ldots, I_k$ such that $g(x,y)$ is constant
on each ``rectangle'' $I_i \times I_j$.  
We refer to the sets $I_i$ as {\bf podes}. 
If $g$ is $k$-podal for some integer $k$, we say that $g$ is {\bf multipodal}. We often use the words 
{\bf bipodal} and {\bf tripodal} to mean
2-podal and 3-podal. 
We say that an $(n+m)$-podal graphon has {\bf $(n,m)$ symmetry} if
$g$ is invariant under permutation of $n$ of the podes and is also invariant 
under permutation of the remaining $m$ podes. 
A graphon with $(2,0)$ symmetry is said to be {\bf symmetric bipodal}.


A {\bf phase} is a connected open set in the interior of the Razborov triangle (Figure \ref{fig:phase_space}) in which the reduced optimal graphon
is unique and is a real analytic function of $(e,t)$ in
the following sense. 
Within a phase, and for each simple graph $F$, $\tdens_F$ of the
optimal graphon is required to be an analytic function of $(e,t)$.
In practice, the analyticity of all densities 
$\tdens_F$ is proven by first
showing that the optimal graphon is unique and 
multipodal of a certain form and then showing that
the finitely many parameters needed to describe this
multipodal graphon are analytic functions of $(e,t)$. 

A {\bf phase transition} occurs where some $\tdens_F$ is not
analytic or is not even defined, such as where the optimal graphon is not unique. 
Phase transitions have only been shown to occur on boundaries of phases, on curves.

\subsection{\label{detailed results}{Detailed results}} The following theorems give a simple description
of what happens near
almost all points along the boundary of the Razborov triangle, Figures \ref{fig:phase_space} and \ref{fig:conjecture}. We prove that the unique optimal graphons in the phases near the boundary are multipodal. The
cited theorems in later sections include additional estimates on 
how the parameters of the optimal graphons scale as the constraints approach the boundary. 

\begin{theorem}[Theorem \ref{thm:flat2}] \label{thm:flat}
For each fixed $e<1/2$ and all $t$ sufficiently small, the 
optimal (reduced) graphon in $\reduced_{e,t}$ is unique and is symmetric bipodal, with parameters that vary analytically with $(e,t)$.
\end{theorem} 

\begin{theorem}[Theorem \ref{thm:scallop2}] \label{thm:scallop}
Let $n \ge 1$ be an integer. For every $e \in \left ( \frac{n}{n+1}, \frac{n+1}{n+2}\right )$, with 
corresponding minimal triangle density $t_0$ (depending on $e$), 
and for all
$\Delta t$ sufficiently small, the optimal (reduced) graphon in $\reduced_{e, t_0 + \Delta t}$
is unique and $n+2$-podal, with $(n,2)$ symmetry and with parameters that vary analytically with $(e,t)$. 
\end{theorem}

Note that these theorems do not make any claims about what happens exactly over
the cusps, i.e.~when 
$e = \frac{n}{n+1}$. When $e=1/2$ ($n=1$), the optimal
graphon has long been known to have a symmetric bipodal structure. 
When $n$ is larger, the optimal graphon is believed to have $(n+1,0)$ symmetry. 
However, a small neighborhood of each cusp is believed to intersect four(!) 
different phases, making a precise characterization difficult. 

\begin{theorem}[Theorem \ref{thm:orderparameter2}] \label{thm:orderparameter}
All of the phases above the scallops proven in Theorems \ref{thm:flat} and \ref{thm:scallop} have unique optimal reduced graphons with distinct symmetries and cannot be analytically continued to one another.
\end{theorem}

In the notation of Figure \ref{fig:conjecture}, Theorem \ref{thm:flat}
proves that the region just above the flat part of the bottom boundary
is part of the $A(2,0)$ phase.  
Theorem \ref{thm:scallop} 
proves the existence of all of the $C(n,2)$ phases, and Theorem 
\ref{thm:orderparameter} shows that these phases are all different. 

\begin{theorem}[Theorem \ref{thm:top2}] \label{thm:top}
For each fixed $e \in (0,1)$ and all $t$ sufficiently close to (but below) $e^{3/2}$, the 
optimal graphon with edge/triangle densities $(e,t)$ is unique and bipodal, with parameters that vary analytically with $(e,t)$. 
\end{theorem}

\noindent That is, the region just below the upper boundary is part of a bipodal phase. There is every reason to
believe that this is part of the same bipodal $F(1,1)$ phase that is found just above the ER curve, but this has
not yet been proved. 

All of these theorems can be viewed as extensions of extremal graph theory. 
Pikhurko and Razborov's results \cite{PR} determine {\em unique, entropy-optimal} graphons on the boundary of Figure \ref{fig:phase_space} \cite{RS1}. Theorems \ref{thm:flat}--\ref{thm:top} 
describe the infinite number of distinct neighboring phases. 

\section{Lagrange multipliers and the ``\Value''
functional}\label{sec:Lagrange}

In this section we develop the concept of ``\Value''. This is a
quantity associated with columns of a graphon, that is with the functions 
$g_x(y) = g(x,y)$ of one variable defined by fixing $x$ and allowing $y$ to float. We will see 
in Theorem \ref{thm:worth} that the
columns of an entropy optimal graphon must maximize \Value. Working with
columns, rather than just with the value $g(x,y)$ of the graphon $g$ at
each point $(x,y)$ separately, gives us the analytical control needed to prove
our main theorems.

To accomplish this we need a theory of Lagrange
multipliers.
For these purposes, we treat graphons as elements of $L^2([0,1]^2)$ with the $||\cdot||_2$-norm. The $L^2$-topology is finer than 
the topology defined by the cut distance. All $L^2$-limits are limits in the cut metric, but not vice-versa.
In order to combine our machinery with established theorems about reduced graphons, we will eventually 
need to prove that certain sequences converge in $L^2$ and not just in the cut distance. 

We develop Lagrange multipliers in several steps. First we develop the theory for a 
certain class of variations where the standard theory of functional derivatives applies. This determines our
Lagrange multipliers $(\alpha, \beta)$. We then show that, for other $L^2$-small changes to an optimal graphon, a certain
functional involving $\alpha$ and $\beta$ cannot increase to leading order in the size of the change. We compute the
effect of changing a small set of columns of our graphon and express the difference in this functional in terms of the 
\Values{} of the old and new columns. Since the functional cannot increase, we conclude that the columns of the original
(optimal) graphon must all maximize \Value. 

The simplest changes involve varying the value of our graphon $g$ gradually at each point, as indicated
by a bounded symmetric function $g_1: [0,1]^2 \to \R$. We consider a family of graphons $g_s$ of the form 
\be g_s(x,y) = g_0(x,y) + s g_1(x,y). \label{eq:type-a} \ee
To remain in $\regular$ we must have $0 \le g_0(x,y) + s g_1(x,y) \le 1$ for all sufficiently small $s$.
This can always be achieved by choosing $g_1$ to be supported on a subset of 
$[0,1]^2$ on which $g_0(x,y)$ is bounded away from 0 and 1. For now, we do not consider variations that change 
$g(x,y)$ at points where $g_0(x,y)=0$ or 1. 

The resulting changes to the edge density, triangle density and Shannon entropy are:
\begin{eqnarray} 
\Delta \edens := \edens(g_s)-\edens(g_0) & = & s \iint g_1(x,y)\, dx \, dy + o(s), \cr
\Delta \tdens := \tdens(g_s)-\tdens(g_0) & = & s \iint 3 G(x,y) g_1(x,y)\, dx \, dy + o(s), \cr
\Delta S := S(g_s)-S(g_0) & = & s \iint H'(g_0(x,y)) g_1(x,y)\, dx \, dy + o(s), 
\end{eqnarray}
where 
\be G(x,y) = \int_0^1 g_0(x,z) g_0(y,z) \, dz, \ee 
\noindent and $H'(u)= \ln \left (\frac{1-u}{u} \right )$ is the derivative of $H$ defined in equation (\ref{entropy function}).
The quantities $H'(g_0(x,y))$, $1$ and $3 G(x,y)$ are called the {\em functional derivatives} of 
$S(g)$, $\edens(g)$, and $\tdens(g)$ with respect to $g$ at $g=g_0$ and are sometimes denoted 
\be \frac{\delta S(g)}{\delta g(x,y)}, \qquad \frac{\delta \edens(g)}{\delta g(x,y)}, \qquad 
\hbox{and} \qquad \frac{\delta \tdens(g)}{\delta g(x,y)}. \ee
These functional derivatives 
are elements of 
$L^2(A)$, where $A$ is the subset of $[0,1]^2$ where $0<g(x,y)<1$, and are only defined up to sets of measure zero. 

\begin{remark} The graphon $g$ is an element of $L^2([0,1]^2)$ and so 
is not 
literally a function. Rather, it is an equivalence class of functions that agree off  
sets of measure zero. Similarly, the ``function'' $G(x,y)$, which gives the inner product of the 
columns $g_x$ and $g_y$, is only defined up to
sets of measure zero. If $\bar g: [0,1]^2 \to
[0,1]$ is a symmetric function that agrees with $g$ apart from a negligible set, then Fubini's Theorem
says that, for almost every $x$,
the columns $g_x$ and $\bar g_x$ agree except on a negligible subset of $[0,1]$,
and in particular represent the same ``function'' in $L^2([0,1])$.
So for almost every pair $(x,y) \in [0,1]^2$, the inner product of $g_x$ and $g_y$
is the same as the inner product of $\bar g_x$ and $\bar g_y$. That is, the 
functions $G(x,y)$ computed from $g$ and $\bar g$ agree except on a set of 
measure zero. 
\end{remark} 

\begin{lemma} \label{lem:independent} If $g$ is not a constant graphon, then $G(x,y)$ is not constant.  \end{lemma} 
\begin{proof} By Cauchy-Schwarz, 
\be G(x,y) \le \sqrt{G(x,x)G(y,y)} \le \max(G(x,x), G(y,y)). \ee
Proving by contradiction, the only way for $G(x,y)$, $G(x,x)$ and $G(y,y)$ to all be equal is if the columns of $g$ at $x$ and $y$ are identical. But if all the columns of $g$ are identical (and likewise all of the rows, since $g$ is symmetric), then $g$ is a 
constant graphon.
\end{proof}

\begin{theorem} \label{thm:Lagrange} 
Let $g$ be an entropy-maximizing graphon, subject to the constraints $\edens(g)=e$ and $\tdens(g)=t$.
If the function $G(x,y)$ is not constant on the set of points $(x,y)$ where 
$0 < g(x,y) < 1$, then there exist unique Lagrange multipliers $\alpha$ and $\beta$ such that 
\be \label{eq:EL1} H'(g(x,y)) = \alpha + \beta G(x,y) \ee
almost everywhere. Furthermore, the function $g(x,y)$ is bounded away from 0 and 1. 
\end{theorem}

\begin{proof} 
Since $G(x,y)$ is not constant on the set $0<g(x,y)<1$, $G(x,y)$ is not constant  on the set $\epsilon < g(x,y) < 1-\epsilon$ for all sufficiently small $\epsilon$, so the functional derivatives of $\edens$ and $\tdens$ are 
linearly independent on this set, so 
we can vary $\edens$ and $\tdens$ independently by choosing appropriate functions $g_1(x,y)$ that are supported 
on this set. Similarly, if the three functions $\frac{\delta \edens}{\delta g}$, $\frac{\delta \tdens(g)}{\delta g(x,y)}$, and $ \frac{\delta S(g)}{\delta g(x,y)}$ were linearly independent on the set of points where $0<g(x,y)<1$, then we could vary $\edens$, $\tdens$,
and $S$ independently, and in particular we could increase $S$ while keeping $e$ and $t$ exactly constant. Since this
contradicts the optimality of $g$, we conclude that $ \frac{\delta S(g)}{\delta g(x,y)}$ is a linear combination of 
$\frac{\delta \edens}{\delta g}$ and $\frac{\delta \tdens(g)}{\delta g(x,y)}$, so we can write
\be \label{eq:Lagrange-abstract} 
\frac{\delta S}{\delta g(x,y)} = \alpha \frac{\delta \edens}{\delta g(x,y)} + \frac{\beta}{3} \frac{\delta \tdens}
{\delta g(x,y)}. \ee
Since 1 and $G(x,y)$ are linearly independent, the coefficients $\alpha$ and $\beta$ are uniquely defined. 
Plugging the functional derivatives into equation (\ref{eq:Lagrange-abstract}) then gives (\ref{eq:EL1}),
which applies at points where $0<g(x,y)<1$. 

Since $\alpha$ and $\beta$ are finite and since $G(x,y)$ is bounded, $H'(g(x,y))$ is bounded, so $g(x,y)$ cannot be 
arbitrarily close to 0 or 1. Either $g(x,y)$ equals 0 or 1 (in which case equation (\ref{eq:EL1}) does not apply) or 
$g(x,y)$ is bounded away from 0 or 1. 

We now eliminate the first possibility. If $g(x,y)=0$ or 1 on a set of positive measure, we can change the value of 
$g(x,y)$ at such points away from 0 and 1, thereby increasing $S$ greatly while only changing $\edens(g)$ and $\tdens(g)$
slightly. Since we can vary $\edens$ and $\tdens$ independently with an appropriate choice of $g_1$, we can restore
the original values of $\edens$ and $\tdens$ with the resulting change in $S$ being governed by (\ref{eq:Lagrange-abstract}).
The total effect is to increase $S$ while leaving $\edens$ and $\tdens$ fixed, which is a contradiction. 
\end{proof} 

The crux of the proof is the same as the derivation of Lagrange multipliers in finite dimensions. The fact that 
linearly independent derivatives give us the ability to vary $\edens$ and $\tdens$ (and possibly $S$) independently
is just the inverse function theorem in $\R^2$ and $\R^3$. The treatment of the points where $g(x,y)=0$ or 1 is
essentially the same as what we do to maximize a function on the boundary of a domain in $\R^n$. Our setting is the 
infinite-dimensional space of graphons, but the core arguments are just finite-dimensional calculus. 

Theorem \ref{thm:Lagrange} was stated in terms of functional derivatives, but we can also 
speak in terms of small changes to a graphon. As long as the 
functional derivatives of $\edens$ and $\tdens$ are linearly independent on the set of points where $0<g(x,y)<1$, 
we can vary $\edens$ and $\tdens$
independently by choosing an appropriate $g_1(x,y)$, with the resulting change in $S$ being given by 
\be \label{eq:DeltaS} \Delta S = \alpha \Delta \edens + \frac{\beta}{3} \Delta \tdens + 
o(s\|g_1\|_2). \ee

\begin{remark}
But what if a graphon $g$ fails to meet the assumptions of Theorem \ref{thm:Lagrange}, with $g(x,y)$ equaling 0 or 1 on part of the unit square and $G(x,y)$ being constant everywhere else? That is, what if we cannot vary $\edens$
and $\tdens$ independently at $g$?
We call such graphons {\bf singular entropy maximizers}. 
These certainly appear on the boundary of the Razbarov triangle. In the interior they appear on 
the \ER{} curve $t=e^3$ and may appear elsewhere, 
so sometimes extra work is needed to exclude them. 
(See \cite[Theorem 4.1]{KRRS1} for such an argument at a specific phase transition.) 

Fortunately for us, such singular graphons do not appear in any of the regions being studied in this paper. 
We postpone 
an explanation of this fact to the end of this section and return to the more typical situation of 
{\em nonsingular} entropy maximizers that satisfy the assumptions of Theorem \ref{thm:Lagrange}. 
\end{remark}

We now consider infinitesimal changes to a graphon obtained by changing $g(x,y)$ by a ``macroscopic 
amount'' on a set of infinitesimal measure $s$.  For instance, we might vary the sizes of the podes in 
a multipodal graphon. Such changes are not covered by Theorem \ref{thm:Lagrange}. 
Nonetheless, they satisfy Lagrange-like {\em inequalities}. 

\begin{proposition} \label{prop:EL-inequality} Suppose that $g_0$ is a non-constant and non-singular
entropy maximizer with Lagrange multipliers 
$\alpha$ and $\beta$. Let $g_s$ be a family of graphons obtained by changing $g_0$ on sets of measure 
$s$. Then 
\be \Delta S \le \alpha \Delta \edens + \frac{\beta}{3} \Delta \tdens + o(s), \label{eq:EL-inequality} \ee
where $\Delta S = S(g_s)-S(g_0)$, $\Delta \edens = \edens(g_s)-\edens(g_0)$, and $\Delta \tdens
= \tdens(g_s)-\tdens(g_0)$. 
\end{proposition} 

\begin{proof} Supposing this to be false, we will construct a variation of $g_0$ that has the same edge 
and triangle densities but more entropy. As previously noted, we can adjust $\edens$
and $\tdens$ independently by adding a function supported on $\epsilon < g(x,y) < 1-\epsilon$, with the resulting
changes in $S$ given by equation (\ref{eq:DeltaS}). Applying these changes to $g_s$ to restore the original values
of $\edens$ and $\tdens$, we should get a change in $S$ that is at least $\Theta(s)-o(s)>0$, which is a contradiction. 

To complete this argument, we must bound the cross terms 
from adding a function of pointwise size $O(s)$ to $g_0$ and in changing $g_0$ on a set of measure $s$. 
Changing $g(x,y)$ by $O(s)$ can only change $H(g(x,y))$ by $O(s \ln(s))$ and can only change
$g(x,y)g(y,z)g(x,z)$ by $O(s)$. Integrating over a region of size $O(s)$, this can
change the entropy by $O(s^2 \ln(s))$, the edge density by $O(s^2)$ and the triangle density by $O(s^2)$, 
resulting in an $o(s)$ change to $S - \alpha \edens - \frac{\beta}{3} \tdens$, which can be absorbed into the
$o(s)$ error term in equation (\ref{eq:DeltaS}).
\end{proof}

The estimate (\ref{eq:EL-inequality}), combined with the theory of Lagrange multipliers for pointwise-small changes (Theorem \ref{thm:Lagrange} as summarized in equation (\ref{eq:DeltaS}))
can be described in terms of a functional
\be \label{free energy} F(g) = S(g) - \alpha \edens(g) - \frac{\beta}{3} \tdens(g). \ee 
If $g$ is a nonsingular constrained entropy maximizer, then pointwise small changes to $g$ cannot change $F$
to first order, while
macroscopic changes on small sets can decrease $F$ but cannot increase $F$ to first order.

The functionals $S$ and $\edens$ are local, in that there is a contribution
from each point $(x,y)$ in $[0,1]^2$ and we integrate the local contributions to
get the global quantity. If the triangle density were also local, 
then $F$ would be the integral of a local
density and our variational equations would come from setting the
derivative of this density with respect to $g(x,y)$ equal to zero. 
That is the typical situation when doing calculus of variations, 
especially in classical and quantum field theory, with 
the global action being the integral of a local Lagrangian density \cite{BS}.

However, the triangle density \be \tdens(g) = \iiint g(x,y) g(y,z)
g(z,x) \, dx \, dy \, dz \ee is {\em not} local. It involves
interactions between the values of $g$ at  the three points  $(x,y)$, $(x,z)$ and $(y,z)$.
To accommodate this complication it is useful to consider macroscopic changes to entire columns. 
We define a quantity that measures the effect of such changes.

\noindent {\bf Definition of \Value{}:} {\em Let $C$ be a possible column of a graphon $g$. That is, $C: [0,1] \to 
[0,1]$ is a Borel measurable function. The {\em \Value} of
$C$ is \be W(C) = 2 \int_0^1 H(C(y)) \, dy -
2 \alpha \int_0^1 C(y)\, dy - \beta \iint C(y) C(z) g(y,z) \, dy\, dz.
\label{eq:value} \ee
Note that this depends explicitly on the graphon $g$ as well as on $C$ and the 
Lagrange multipliers $\alpha$ and $\beta$. }
\medskip

\begin{proposition} Let $(\alpha, \beta)$ be specified and suppose that $\tilde g(x,y) = g(x,y)$ except when $x$ or $y$
lies in a set $I$ of measure $s$. Then 
\be \label{eq:DeltaF} F(g) - F(\tilde g) = \int_I W(g_x)-W(\tilde g_x) \, dx + O(s^2), \ee
where $g_x$ and $\tilde g_x$ are columns of $g$ and $\tilde g$. 
\end{proposition}

\begin{proof} The quantities $\edens(g)$, $S(g)$ and $\tdens(g)$ are all double or triple integrals over 
$[0,1]^2$ or $[0,1]^3$. The integrands for $g$ and $\tilde g$ are identical except where one of the variables lies in 
$I$. To compute $F(g)-F(g')$, we must only keep the contributions of $x \in I$, multiply by 2 or 3 to allow for
the similar contributions of $y \in I$ or $z \in I$, and make adjustments for where two or three variables are in $I$.
Since the integrand is bounded and the set of points where multiple variables lie in $I$ only has measure $O(s^2)$, we
can compute $F(g)-F(\tilde g)$ to within $O(s^2)$ by assuming that $x \in I$ and leaving $y$ and $z$ free. 

We begin with the edge density: 
\be \edens(g) - \edens(g') = 2 \int_{x \in I} \int_{y\in[0,1]} g_x(y)-\tilde g_x(y) \,dy \, dx + O(s^2). \ee
The entropy is similar:
\be S(g) - S(\tilde g) = 2 \int_{x \in I} \int_{y\in[0,1]} H(g_x(y))-H(\tilde g_x(y)) \, dy \, dx + O(s^2). \ee
The triangle density has a prefactor of 3 instead of 2 because it is a triple integral:
\begin{eqnarray} \tdens(g)-\tdens(\tilde g) &=& 3 \int_{x \in I} \iint_{y,z\in[0,1]} g_x(y)g_x(z)g(y,z)-\tilde g_x(y)\tilde g_x(z)\tilde g(y,z) \,dy \, dz \, dx + O(s^2) \cr 
 &=& 3 \int_{x \in I} \iint_{y,z\in[0,1]} \left (g_x(y)g_x(z)-\tilde g_x(y)\tilde g_x(z)\right ) g(y,z) \,dy \, dz \, dx + O(s^2), 
 \end{eqnarray}
 where in the last line we used the fact that $\tilde g(y,z)$ only differs from $g(y,z)$ when $y\in I$
 or $z \in I$. 
 Multiplying the change in $\edens$ by $-\alpha$ and the change in $\tdens$ by $-\beta/3$ and adding terms, we get
 (\ref{eq:DeltaF}). 
 \end{proof}

\begin{theorem}  \label{thm:worth} If $g$ is a nonsingular entropy maximizer with Lagrange multipliers 
$(\alpha, \beta)$ then $g$ agrees off a set of measure zero with a graphon where every
column maximizes $W$.  In particular, every column must have
the same \Value. \end{theorem}

\begin{proof} First we show that almost every column maximizes \Value. 
Since we are working in $L^2$, where 
functions that differ on sets of measure zero are considered equivalent, we can get an equivalent representative by
replacing any columns that don't maximize \Value{} with ones that do.

Let $W_{max}$ be the supremum of the {\Values} of all possible columns. 
If there is a set of columns of positive measure whose {\Values} are strictly less than $W_{max}$
then for some $\delta>0$ there is a set of columns of positive measure whose {\Values}  are all bounded above by $W_{max}-2\delta$. 
Replacing a subset $I_s$ of measure $s$ of such columns with a column $C$ whose {\Value} is within 
$\delta$ of $W_{max}$ will increase $F$ by at least 
$s\delta + O(s^2)$, contradicting Proposition \ref{prop:EL-inequality}. 

(Note that ``change a set of columns to $C$'' leads to an ambiguity for $g(x,y)$ when both $x$ and $y$ are in $I_s$. We can resolve this 
ambiguity by setting $g(x,y)=1$ on $I_s \times I_s$, or by setting $g(x,y)=0$, or by picking any other symmetric 
function in this square. Since the square where the ambiguity occurs has area $s^2$, different choices will yield 
values of $F$ that differ by $O(s^2)$, which does not affect the violation of Proposition \ref{prop:EL-inequality}.)
\end{proof}

Having established the variational equations for nonsingular entropy maximizers, both for points with 
Theorem \ref{thm:Lagrange} and for columns with Theorem \ref{thm:worth}, we return to the (non)existence of 
singular entropy maximizers. If $g$ is a singular entropy maximizer, 
we call $\tdens(g)$ a 
{\em singular value of $t$} for the given edge density $e=\edens(g)$.
We claim that singular $t$'s are too sparse to matter. We begin with their measure.  

\begin{lemma}\label{lem:nonsingular}  
For each $e \in (0,1)$, the set of singular $t$-values has measure zero. \end{lemma}

\begin{proof} As explained below, the Boltzmann entropy function $\B(e,t)=\max_{g\in \regular_{e,t}} S(g)$ 
for fixed $e$ is monotonically increasing
in $t$ on $(t_{min}, e^3)$ and monotonically decreasing on $(e^3, e^{3/2})$. This 
makes $d\B$ a finite measure on $(t_{min},e^3)$ and makes $-d\B$ a finite measure on $(e^3,e^{3/2})$.
We will show that the measure is singular at all singular $t$-values. The theorem then follows from the 
fact that the support of the singular part of a finite measure on an interval has zero Lebesgue measure. 

To see the monotonicity of $\B$, let $g_0$ be an entropy maximizer at $(e,t)$ with $t<e^3$
and consider the family of graphons 
\be \label{eq:g-s} g_s(x,y) = se + (1-s) g_0(x,y), \ee
which is defined for all $s \in [0,1]$. $S(g_s)$ is an increasing function of $s$ (thanks to the 
concavity of $H(u)$) and in particular $S(g_s)>S(g_0)$ for all $s>0$. 
Since $\tdens(g_s)$ goes from $t$ to $e^3$ as $s$ goes from 0 to 1, 
and since $S(g_s)$ is a lower bound for $\B(e, \tdens(g_s))$, 
$\B(e, t') > \B(e,t)$ for every $t' \in (t,e^3)$. 

We also note that $\tdens(g_s)>t$ for all $s>0$, since otherwise, by the 
intermediate value theorem, there would be a positive $s$ with $\tdens(g_s)=t$. Since $S(g_s)>S(g_0)$,
as we just discussed, that would contradict the optimality of $g_0$. 

Now suppose that $t$ is singular, with a singular entropy maximizer $g_0$. 
Since $g_0(x,y)$ equals 0 or 1 on a set of positive measure, $S(g_s)-S(g_0)$ scales as $s \ln(1/s)$ as $s 
\to 0$. However, $\tdens(g_s)$ is
a polynomial in $s$ and cannot grow faster than linearly for small $s$. Thus 
\be \lim_{s \to 0^+} \frac{S(g_s)-S(g_0)}{\tdens(g_s)-\tdens(g_0)} = +\infty.  \ee
Since $\B(e,t)=S(g_0)$, and since $S(g_s)$ is a lower bound for $\B(e, \tdens(g_s))$, 
$\B$ must be
increasing infinitely fast at $t$. That is, $t$ is a singular point of the measure $d \B$. 

The exact same arguments work for $t>e^3$, only with $\tdens(g_s)$ being a decreasing function of $s$,
with $\B$ being a decreasing function of $t$,
and with $\B$ decreasing at infinite rate at singular $t$-values.  \end{proof} 

We note that Lemma \ref{lem:nonsingular} proves there are no singular entropy maximizers in the regions studied in this paper.

\begin{theorem} \label{thm:no-singular} 
Fix the edge density $e$ and consider an open interval $I$ of triangle densities $t$. Suppose that there is a
differentiable function $f(t)$, defined for all $t \in I$, that equals the Boltzmann entropy $\B(e,t)$ 
whenever $t$ is non-singular. Then every $t \in I$ is non-singular. 
\end{theorem} 

\begin{proof} The Boltzmann entropy is never 
differentiable at $t=e^3$ \cite{RS3}, so we only need to consider intervals $I$ that are either above or
below the \ER{} curve. By Lemma \ref{lem:nonsingular}, the set of singular $t$-values in $I$ has measure zero, so
the complement of that set is dense in $I$. 
Since $\B(e,t)$ is monotonic on $I$ and equals a (differentiable and therefore) continuous function $f(t)$ on
a dense subset of $I$, it must equal $f(t)$ on all of $I$. But then $\B(e,t)$ is differentiable in $t$ for
all $t \in I$, so there are no singular $t$-values in $I$. 

\end{proof}

Theorems \ref{thm:worth} and \ref{thm:no-singular} give us a strategy by which we can now prove (with much work!) Theorems
\ref{thm:flat}, \ref{thm:scallop} and \ref{thm:top}. In each case, we use variational equations 
that apply whenever $t$ is 
non-singular. For all such $t$, we show that the optimal graphon must take a certain form, with an entropy that is 
(the restriction of) a smooth function of $t$. Theorem \ref{thm:no-singular} then implies that there are no 
singular $t$'s and that 
our calculations apply to all $t$. That is, while we cannot exclude singular entropy maximizers a priori, 
we are able to exclude them a posteriori. 

We note that the analysis of optimizing the function 
$F$ in equation (\ref{free energy}) is the starting 
point of our discussion of ERGMs in Section \ref{sec:No-ERGM}.
\section{Proof of Theorem \ref{thm:flat}} \label{sec:flat}

We now prove a slightly more quantitative version of 
Theorem \ref{thm:flat}: 
\begin{theorem} \label{thm:flat2}
For each fixed $e<1/2$ and all $t$ sufficiently small, the 
optimal graphon with edge/triangle densities $(e,t)$ is unique and is symmetric bipodal, with parameters that vary analytically with $(e,t)$. As $t \to 0$,
the increase $\Delta \B = \B(e,t)-\B(e,0)$ in the 
Boltzmann entropy scales as $t \ln(1/t)$ and the Lagrange multiplier 
$\beta$ scales as $\ln(1/t)$. 
\end{theorem} 

\subsection{Strategy}
The proofs of Theorems \ref{thm:flat2}, \ref{thm:scallop2} and \ref{thm:top2}
all follow the same general outline. 
We will present the proof of Theorem \ref{thm:flat2} in
full detail. The subsequent proofs of Theorems \ref{thm:scallop2} 
and \ref{thm:top2} will 
be somewhat abbreviated, concentrating on what is different in those
cases.

Using the fact that the unique entropy maximizing graphon $g_0$
at $(e,0)$ is symmetric bipodal, we show that the 
optimizing graphons at points near the boundary have the same general structure
away from an exceptional set of small area. Specifically, we partition the unit interval 
into subsets $I_1$, $I_2$ and $I_3$ such that the columns $g_x$ of 
the optimal graphon are $L^2$-close to the columns of the
first pode of $g_0$ when $x \in I_1$ and are $L^2$-close to the columns 
of the second pode of $g_0$ when $x \in I_2$, and where $I_3$ has 
small measure. At this stage, we do not have any control over $g_x$
when $x \in I_3$. 

Knowing the columns $g_x$ when $x \in I_1 \cup I_2$ 
(to within a small error in $L^2$) gives us pointwise control of the function 
$G(x,y)$ on $(I_1 \cup I_2) \times (I_1 \cup I_2)$. The
Euler-Lagrange equations (\ref{eq:EL1}) then give us pointwise estimates of 
$g(x,y)$ in each of the four main rectangles. 

We then study the \Value{} functional $W(C)$. The dependence of this functional on the graphon $g$ comes via
the integral $\iint C(y) C(z) g(y,z) \, dy \, dz$. Since $C$ is bounded,
and since we know $g(y,z)$ away from a set of small measure, we have 
good control over $W(C)$. 
We determine that a \Value-maximizing column can only take one of two approximate forms, namely those 
exhibited by $g_x$ for $x \in I_1$ and for $x \in I_2$. We then reassign each point $x \in I_3$ to $I_1$ 
or $I_2$, depending on which \Value-maximizing form $g_x$ takes. 
The result is then a graphon with two (approximate) podes. 

Using the pointwise equations (\ref{eq:EL1}), we bound the variation of $g(x,y)$ in each rectangle $I_i \times I_j$ 
by a multiple of the
variation in a neighboring rectangle. Combining these results, the variation in each 
rectangle is bounded by a small multiple of itself, and so must 
be zero. That is, our optimal
graphon must be exactly bipodal.

The space of bipodal graphons with given values of $(e,t)$ is only 2-dimensional. Using ordinary 2-dimensional
calculus, we determine that the entropy $S(g)$ is maximized when the graphon is symmetric. 

To account for the fact that the Lagrange multipliers $\alpha$ and $\beta$ are only defined 
for almost every $t$ and not necessarily for every $t$, we analyze optimal graphons in two passes. In the first 
pass we use the \Value{} function and pointwise equations (\ref{eq:EL1}), as outlined above, to establish that the optimal graphon is symmetric bipodal for almost every 
$t$ that is sufficiently small. This is the bulk of the proof of Theorem \ref{thm:flat2}. 

This shows that the Boltzmann entropy is almost everywhere 
equal to the Shannon entropy of a symmetric bipodal graphon, which is a differentiable function of $t$ for 
each $e$. 
By Theorem \ref{thm:no-singular}, this then implies that there are no singular values of $t$ and that our arguments apply at 
{\em all} sufficiently small values of $t$.

The same two-pass argument applies to the proofs of Theorems \ref{thm:scallop2} and \ref{thm:top2}, only replacing
``symmetric bipodal'' with the particular graphon symmetry described in those theorems. Specifically,
in the first pass we show that the Boltzmann entropy is almost everywhere equal to the maximum Shannon entropy
among multipodal graphons of a certain sort. The solution to the resulting finite-dimensional maximization problem
yields a smooth function of $t$. Theorem \ref{thm:no-singular} then says that our results apply at every $t$.

\begin{remark} As is standard when working with $L^2$ spaces, the proofs of Theorems 
\ref{thm:flat2}, \ref{thm:scallop2} and \ref{thm:top2} are written as if our graphons were actual functions
on $[0,1]^2$. But in fact they are equivalence classes of functions that agree away from a set of measure
zero. This has several consequences, none of which materially affect the flow of the proofs: 
\begin{itemize}
\item When we use the variational equations (\ref{eq:EL1}), the results apply almost everywhere, not
literally everywhere.  Whenever we use those equations to compute an upper or lower bound on a graphon, that bound should always
be understood to mean ``apart from on a set of measure zero''. 
\item When we speak of the ``maximum'' value of a graphon $g$ on a region, we actually mean the 
essential supremum of the function, namely the smallest number $M$ such that $g(x,y) \le M$ on a set of 
full measure. The ``minimum'' is similar.
\item Since we are working with functions mod sets of measure zero, we are free to change the value of 
our graphon on sets of measure zero whenever we wish. In this way, we could make the variational equations
apply everywhere, or we could make the maximum of a function equal the essential supremum. We 
{\em could}, but we won't actually subject the reader to such painstaking bookkeeping! 
\end{itemize}
Instead, we will {\bf not} keep track of sets of measure zero in these proofs, such as deciding 
whether a 
pode contains its endpoints. All the important properties of graphons (or at least all the properties
considered in this paper) are based on integrals, for which sets of measure zero don't matter at all.
\end{remark}

Having explained the 
process in all three settings, we return our focus to the first pass.

\subsection{Defining approximate podes}

There is a unique entropy maximizer $g_0$ at $(e,0)$ on the bottom boundary of the Razborov triangle
(up to measure-preserving transformations of the 
unit interval, as usual). 
This graphon
is symmetric bipodal, taking values 0 on the diagonal blocks and $2e$ on the off-diagonal blocks. As we
approach the bottom boundary of the Razborov triangle, we claim that 
any sequence $\{g_i\}$ of entropy maximizers must converge (after appropriate measure-preserving transformations) 
to $g_0$ in $L^2$. 

To see this, we invoke the compactness of the space of reduced graphons in the cut metric. A subsequence 
must converge to a limit $[g_\infty]$ in the cut metric. 
Lemma 2.1 in \cite{CV} proves that $S$ is upper-semicontinuous on $\reduced$ in the cut metric $\delta_{cut}$. This implies that the limit of the entropies of the entropy maximizers is at least $S([g_0])$, so we must have 
$S([g_\infty]) \ge S([g_0])$. But $[g_0]$ is the unique entropy maximizing reduced graphon at $(e,0)$, so $[g_\infty] = [g_0]$. 

The entire sequence $\{[g_i]\}$, and not just a subsequence, must converge to $[g_0]$. 
If it did not, we could find a subsequence where all points
were bounded away from $[g_0]$ in the cut metric $\delta_{cut}$. Applying the previous argument to this subsequence would then yield
a contradiction. That is, after applying appropriate measure-preserving transformations of $[0,1]$, the sequence 
$\{g_i\}$ of entropy maximizing graphons must converge to $g_0$ in the cut distance $d_{cut}$. 

Let $I_1=[0,1/2]$ and $I_2=[1/2, 1]$ be the two podes of $g_0$. 
By the definition of the cut distance, the average value of $g_i$ must converge 
to 0 on $I_1 \times I_1$ and $I_2 \times I_2$ and to $2e$ on $I_1 \times I_2$ and $I_2 \times I_1$. The variance of $g_i$ 
must go to zero on each of these rectangles, or else $\lim S(g_i)$ would be strictly less than $S(g_0)$. 
Having $\lim S(g_i) < S(g_0)$ is impossible because there exist explicit symmetric bipodal graphons with $t \to 0$ whose entropies give a lower bound for $S(g_i)$ and whose entropies approach $S(g_0)$ as $t \to 0$. 

Since the mean of $g_i$ in each rectangle approaches the value of $g_0$ and since the variance goes 
to zero, $\{g_i\}$ approaches $g_0$ in $L^2$.
 That is, for every 
$\epsilon >0$ there is a $\delta >0$ such that, for all $t<\delta$ and all optimal graphons $g$ at $(e,t)$, 
applying a measure-preserving transformation of $[0,1]$ to $g$ we have
$\|g-g_0\|_2 < \epsilon$.
(Note that we have not assumed that the optimal graphon $g$  
is unique. That will be proven in due course.) 

Let $g$ be such an optimal graphon for a particular value of  $(e,t)$. Then 
\be \epsilon^2 \ge \int_0^1 dx \int_0^1 dy \big ( g(x,y) - g_0(x,y) \big )^2, \ee
so 
\be \int_0^1  \big ( g(x,y) - g_0(x,y) \big )^2 dy < \epsilon, \ee
except on a set of $x$'s of measure $\epsilon$ or less. 
Call that exceptional set $I_3$. Let $I_1$ and 
$I_2$ be the intersection of $I_3^c$ with $[0, 1/2]$ and $[1/2, 1]$, respectively. Let
$C_1$ and $C_2$ to be the columns of $g_0$ on the two podes, namely $2e$ times the indicator functions 
of $[1/2, 1]$ and $[0,1/2]$, respectively. We have 
broken the unit interval into three pieces $I_1$, $I_2$, $I_3$, such that:
\begin{itemize}
\item For all $x \in I_1$, $\| g_x - C_1\|_2 < \sqrt{\epsilon}$.
\item For all $x \in I_2$, $\| g_x - C_2 \|_2 < \sqrt{\epsilon}$. 
\item When $x\in I_3$ we do not yet have any estimates on $g_x$.
\end{itemize}
We will refer to the sets $I_1$, $I_2$ and $I_3$ as podes, even though we are {\bf not} assuming that
the graphon $g$ is exactly tripodal. 

\subsection{Variational equations} We now use the pointwise variation equations (\ref{eq:EL1}) to get
some preliminary estimates on $g(x,y)$. The first two derivatives of the
function $H$ are
\be H'(u) = \ln(1-u) - \ln(u), \qquad H''(u) = - \left ( \frac{1}{u} + \frac{1}{1-u} \right ). \ee
The quantity $G(x,y)$ is the 
$L^2$-inner product of $g_x$ and $g_y$, which we denote 
$\langle g_x | g_y \rangle$. That is,
\be G(x,y) = \langle g_x | g_y \rangle = \int_0^1 g(x,z) g(y,z) dz. \ee
If $x$ and $y$ are both in $I_1$, 
or both in $I_2$, then $G(x,y) = 2e^2 + O(\sqrt{\epsilon})$. If one is 
in $I_1$ and the other is in $I_2$, then $G(x,y) = O(\sqrt{\epsilon})$. 
If either or both are in $I_3$, then our estimates do not apply. 

For $(x,y) \in I_1 \times I_1$ or $I_2 \times I_2$, we have 
\be H'(g(x,y)) = \alpha + 2 \beta e^2 (1 + O(\sqrt{\epsilon})) = 2 \beta e^2 (1 + O(\sqrt{\epsilon})), \ee
so 
\be g(x,y) = \exp(-2e^2 \beta(1+O(\sqrt{\epsilon})). \ee
(Since $\beta$ is divergent as $t \to 0$ but $\alpha$ is not, we can absorb $\alpha$ into the 
$O(\beta \sqrt{\epsilon})$ error.) 
This means that the contribution of $g(x,y)$ in $I_1 \times I_1$ or $I_2 \times I_2$ to $\beta G$ goes as 
$\beta$ times a negative exponential in $\beta$, and thus has an 
extremely small effect on the value of 
$g(x,y)$ in $I_1 \times I_2$ or $I_2 \times I_1$. 

However, we cannot yet precisely estimate 
$g(x,y)$ in those regions because $G(x,y) = \langle g_x | g_y \rangle$ also gets a contribution, potentially
of order $\epsilon$, from $z \in I_3$.

\subsection{Maximizing \Value{} and eliminating $I_3$}

Let $C: [0,1] \to [0,1]$ be a function whose \Value{} we wish to 
estimate. Let 
\be a = 2 \int_0^{1/2} C(y)\, dy, \qquad b = 2 \int_{1/2}^1 C(y) \, dy. \ee
That is, $a$ and $b$ are the average values of $C$ on $[0,1/2]$ and 
$[1/2,1]$. 

We now consider the three expressions that contribute to $W(C)$:
\begin{itemize}
\item The entropy term $2\int_0^1 H(C(y)) \, dy$ is bounded above by 
$H(a)+H(b)$, thanks to $H''$ being everywhere negative. 
\item The term $-2\alpha \int_0^1 C(y) \, dy$ is exactly $-\alpha (a+b)$. 
\item The term $-\beta \iint C(y) C(z) g(y,z) \, dy \, dz$
is approximately $-e\beta ab$. 
\end{itemize}

Recall that the \Value{} of a column is 
\be W(C) = 2 \int_0^1 H(C(y)) \, dy -
2 \alpha \int_0^1 C(y)\, dy - \beta \iint C(y) C(z) g(y,z) \, dy\, dz.
\label{eq:value2} \ee
If $g$ were equal to $g_0$, maximizing $W(C)$ would involve taking $C(y)$ to be constant on 
$[0,1/2]$ and constant on $[1/2,1]$ and choosing $a$ and $b$ 
to maximize 
\begin{equation} H(a) + H(b) - \alpha (a+b) - e \beta ab. 
\label{eq:ELab} \end{equation}
(Because of the small differences between $g$ and $g_0$, this procedure only gives approximate \Value{}
maximizers, but that is enough for our purposes.) 

Setting the derivatives of (\ref{eq:ELab}) to zero gives the equations
\be H'(a) = \alpha + \beta eb, \qquad H'(b) = \alpha + \beta ea. \ee
Since there is a \Value-maximizer with $a$ close to 0 and $b$ close 
to $2e$ (namely any column with $x \in I_1$), and another 
\Value-maximizer with $a$ close to $2e$ and $b$ close to 0, $\alpha$
must be close to $H'(2e)$, while $\beta$ is large and positive. 

If $a$ is substantially nonzero (say, bigger than $1/\sqrt{\beta}$),
then $e\beta a$ is gigantic and $b$ is extremely close to zero, being 
$O(\exp(-\sqrt{\beta}))$. This makes  
$e\beta b$ tiny so $H'(a) \approx \alpha \approx H'(2e)$ and 
$a \approx 2e$. Similarly, if 
$b$ is substantially nonzero then $a$ is tiny and $b \approx 2e$.
In both those cases, $W(C) \approx H(2e) - 2eH'(2e) = - 
\ln(1-2e)$. The third possibility is that $a$ and $b$ are both 
tiny, but in that case $W(C) \approx 0$, which is strictly less than $-\ln(1-2e)$. 

The upshot is that there are three stationary points of 
(\ref{eq:ELab}) but only two maxima, one that resembles $g_x$ for 
$x \in I_1$ and one that resembles $g_x$ for $x \in I_2$. Since every 
column $g_x$ with $x \in I_3$ must be a \Value-maximizer, and 
since every \Value-maximizer must come close to maximizing (\ref{eq:ELab}), every column $g_x$ with $x \in I_3$ is close in $L^2$ to the columns 
for $x\in I_1$ or $I_2$. We can then 
reassign the points of $I_3$ to $I_1$ or $I_2$ depending on the nature
of $g_x$.

\subsection{Exact bipodality}

Our next step is to upgrade our $L^2$ estimates on the forms of the different 
columns into pointwise estimates. Thanks to each column of $g$ being
$L^2$-close to a column of $g_0$, the function 
$G(x,y) = \langle g_x | g_y \rangle$ is pointwise close to $2e^2$ on 
$I_1 \times I_1$ and on $I_2 \times 
I_2$. By (\ref{eq:EL1}), this forces $g(x,y)$ to be exponentially 
small (specifically, $\exp(-\Theta(\beta)))$ in 
these quadrants. This in turn makes $G(x,y)$ exponentially small on
$I_1 \times I_2$ and $I_2 \times I_1$, which means that $H'(g)$ 
is exponentially close to $\alpha$ in these quadrants, and 
therefore that $g(x,y)$ is pointwise close to constant in these
rectangles. 

We next show that the optimal graphon $g$ 
is {\em exactly} constant on each of those rectangles. 
Let $A$, $B$, and $D$ be the average values of $g(x,y)$ on $I_1 \times I_1$, $I_2 \times I_2$ and 
$I_1 \times I_2$, respectively. Let $\Delta_A$, $\Delta_B$ and $\Delta_D$ be the difference between 
the maximum and minimum values of $g(x,y)$ on those rectangles. Let $c$ be the width of $I_1$.

On $I_1 \times I_1$, the quantity $G(x,y)$ is bounded below by 
\be c (A- \Delta_A)^2 + (1-c) (D - \Delta_D)^2 \ee
and bounded above by 
\be c (A+ \Delta_A)^2 + (1-c) (D + \Delta_D)^2. \ee
The difference between these two expressions is $4cA\Delta_A + 4(1-c)D \Delta_D$. 

All points satisfy the variational equations 
\be H'(g(x,y)) = \alpha + \beta G(x,y). \ee
Subtracting this equation at the smallest value of $G(x,y)$ from that at the largest value, applying 
the mean value theorem to the left hand side, and applying our bounds on the variation in 
$G(x,y)$, we obtain
\be - H''(A_0) \Delta _A \le 4Ac\beta \Delta_A + 4D(1-c)\beta \Delta_D, \ee
where $A_0$ is some number between $A+\Delta_A$ and $A-\Delta_A$.  A little algebra then shows that 
\be \Delta_A \le \frac{4D(1-c)\beta}{-H''(A_0) - 4Ac\beta} \Delta_D \le \frac{3D\beta}{-H''(A)} \Delta_D,\ee
where we have used the difference between 3 and $4(1-c) \approx 2$ to cover for simplifying the 
denominator and replacing $A_0$ with $A$. A similar analysis on $I_2 \times I_2$ shows that 
\be \Delta_B \le \frac{3D\beta}{-H''(B)} \Delta_D. \ee

Meanwhile, on $I_1 \times I_2$, $G(x,y)$ is bounded above and below by 
\be c(A \pm \Delta_A)(D \pm \Delta_D) + (1-c)(B \pm \Delta_B)(D \pm \Delta_D), \ee
where the plus signs give an upper bound and the minus signs give a lower bound. The difference between
the upper and lower bounds is 
\be 2 \beta [ (cA+(1-c)B)\Delta_D + (c \Delta_A + (1-c) \Delta_B) D]. \ee
This implies that 
\be -H''(D_0) \Delta D  \le 2\beta (cA+(1-c)B) \Delta_D + 2 \beta D (c \Delta_A + (1-c)\Delta_B). \ee
A little algebra then gives 
\begin{eqnarray} 
\Delta_D &\le& \frac{2 \beta D (C \Delta A + (1-c)\Delta B)}{-H''(D_0)-2\beta(cA+(1-c)B} \cr 
& \le & \frac{3 \beta D(\Delta_A + \Delta_B)}{-2H''(D)} \cr 
&\le & \frac{9 \beta^2 D^2}{-2H''(D)} \left( \frac{-1}{H''(A)} + \frac{-1}{H''(B)} \right ) \Delta_D. 
\end{eqnarray}

Now recall that $A$ and $B$ are exponentially small in $\beta$ and that 
\be \frac{-1}{H''(A)} = A(1-A) < A \qquad \hbox{and} \qquad \frac{-1}{H''(B)} = B(1-B) < B.\ee
The coefficient of $\Delta_D$ on the right hand side of the last 
line goes to zero roughly as $\beta^2 \exp(-2e^2 \beta)$ as $t \to 0$ and $\beta \to \infty$. Once
the coefficient is less than one, the only solution is $\Delta_D = 0$, which then implies that 
$\Delta_A=0$ and $\Delta_B=0$. In other words, our optimal graphon is exactly bipodal. 

\subsection{Symmetric bipodality}

All that remains is showing that the best bipodal graphon is symmetric, with pode sizes $\frac12$ and 
$\frac12$ and with $A=B$. This requires extensive 
calculations but no sophisticated analysis. Ultimately, it
is just a (grungy) problem in multivariable calculus as follows.

For each triple $(e,t,c)$ we consider the bipodal graphon that maximizes the entropy, subject to the 
constraints that the edge and triangle densities are $(e,t)$ and that the first pode has width $c$. Let
$S(e,t,c)$ be the entropy of this optimal graphon. We must show that this entropy is maximized at 
$c=1/2$. Note that this function is analytic in $c$ for fixed $(e,t)$, insofar as the parameters 
are determined by analytic Euler-Lagrange equations, and is even in $\Delta c := c-\frac12$. 

When $t=0$, the function is easy to compute. The graphon must be zero on $I_1 \times I_1$ and 
$I_2 \times I_2$ and take on the constant value $\frac{e}{2c(1-c)} = \frac{2e}{1-4\Delta c^2}$ on $I_1 \times I_2$. 
The entropy is then 
\begin{eqnarray}  S(e,0,c) &=& \frac12 (1-4 \Delta c^2) H \left ( \frac{2e}{1-4 \Delta c^2} \right ) \cr
& = & S(e,0,1/2) + 2 \ln(1-2e) \Delta c^2 + O(\Delta c^4).
\end{eqnarray}
That is, there is an entropy cost proportional to $\Delta c^2$ associated with having $\Delta c \ne 0$.

Now consider the effect of having $t$ nonzero. Having the graphon nonzero on $I_1 \times I_1$ and $I_2 \times I_2$ provides
additional entropy of order $t \ln (1/t)$. Shifting the value of the graphon on $I_1 \times I_2$ by
an $O(t)$ amount changes the entropy by an additional $O(t)$, but since this is small compared 
to $t \ln(1/t)$, $S(e,t,c)-S(e,0,c)$ is still $O(t \ln(1/t))$. 
In order to overcome the $-2 \ln(1-2e) \Delta c^2$ cost, 
we must have $\Delta c = O(\sqrt{t \ln(1/t)})$. Since $t \sim \exp(-2e^2 \beta)$, this means that $\Delta c$
must be exponentially small in $\beta$ and in particular that $\beta \Delta c$ is a small parameter.

We now compute the quantity $G(x,y)$ in each rectangle and look at the Euler-Lagrange equations for a particular value of $\beta$: 
\begin{eqnarray} 
H'(A) & = & \alpha + \frac{\beta}{2}(A^2 + D^2) - \beta \Delta c (D^2-A^2) \cr
& \approx & \alpha + \frac{\beta}{2} D^2 - \beta \Delta c D^2, \cr 
H'(B) & = & \alpha + \frac{\beta}{2}(B^2+D^2) + \beta \Delta c(D^2-B^2) \cr 
& \approx & \alpha + \frac{\beta}{2}D^2 + \beta \Delta c D^2, \cr 
H'(D) & = & \alpha + \frac{\beta D}{2} (A+B + 2 \Delta c (A-B)) \cr 
& \approx & \alpha, 
\end{eqnarray}
where in our approximations we use the fact that $A$ and $B$ are exponentially small in $\beta$. Since 
$D \approx 2e$, this makes $\alpha \approx H'(2e)$. The terms proportional to $\Delta c$ serve to 
multiply $A$ by a factor of $\exp(-4e^2\beta \Delta c) \approx 1 - 4e^2 \beta \Delta c$ and to multiply
$B$ by a factor of $\exp(4e^2 \beta \Delta c) \approx 1 + 4e^2 \beta \Delta c$. These changes in the values
of $A$ and $B$ (relative to their values when $\Delta c = 0$) slightly change the triangle density 
for a given value of $\beta$, but only by a fraction $O(\beta \Delta c^2)$. Likewise, the contribution 
to the entropy of the $I_1 \times I_1$ and $I_2 \times I_2$ squares changes by a fraction $O(\beta \Delta c^2)$. However, that entropy is only $O(t \ln(1/t))$, so we are dealing with an expression that is  
\be O(\beta t \ln(1/t) \Delta c^2) = O(t (\ln(1/t)^2) \Delta c^2), \ee
since $\beta = O(\ln(1/t))$. This possible entropy gain from having $\Delta c \ne 0$ is much smaller than
the $-2\ln(1-2e) \Delta c^2$ cost, so the optimal value of $\Delta c$ is exactly zero. That is, we 
must have $c=1/2$.  

When $c=1/2$, two of the Euler-Lagrange equations read: 
\begin{eqnarray}
H'(A) & = & \alpha + \beta(A^2 + D^2)/2, \cr 
H'(B) & = & \alpha + \beta(B^2 + D^2)/2.
\end{eqnarray} 
If $A>B$, then the right hand side of the first equation is greater than that of the second, so 
$H'(A) > H'(B)$. But that is a contradiction, since $H'(u) = \ln(1-u) - \ln(u) $ is a decreasing 
function of $u$. Likewise, we cannot have $A<B$. So $A$ and $B$ must be equal, making our optimal graphon 
symmetric bipodal. 

The parameters of a symmetric bipodal graphon are uniquely (and 
analytically) determined by $(e,t)$. 

We also consider how various quantities scale as $t \to 0$. After setting $c=1/2$ and 
$B=A$, a direct calculation shows that 
\be t = \frac{3}{4} AD^2 + \frac14 A^3, \ee
so 
\be A = \frac{4t}{3D^2} + O(t^3) = \frac{t}{3e^2} + O(t^2), \ee
where we have used the fact that $D=2e-A$. Since $A$ was exponentially small in $\beta$, $\beta$
must scale as $\ln(1/t)$. The entropy is 
\be \frac12 (H(A) + H(2e-A)) = \frac12 H(2e) -\frac12 A \ln(A) + O(A), \ee
so $S(g)-\frac12 H(2e)$ scales as $t \ln(1/t)$. 

The Boltzmann entropy $\B(e,t)$ is equal to the 
Shannon entropy $S(g)$ of the optimal graphon at $(e,t)$, so 
$\Delta \B := \B(e,t) - \B(e,0) = S(g)- \frac12 H(2e)$. 
Since $A \approx(t/3e^2)$ and $A \approx \exp(-2e^2\beta)$, 
$\beta \approx \frac{-1}{2e^2} \ln(t/3e^2) \sim \ln(1/t)$. This completes the proof of Theorem \ref{thm:flat2}.
\qed

Note that we have actually proved something slightly stronger than
Theorem \ref{thm:flat2}. We started with an optimal graphon, applied measure 
preserving transformations of $[0,1]$, and wound up with a symmetric bipodal 
graphon. That is, the optimal graphon in $\regular_{e,t}$ is unique up to {\em group}
equivalence. However, a reduced graphon in $\reduced_{e,t}$ is a {\em weak}
equivalence class, with all representatives of this class being constrained 
entropy maximizers. We thus conclude that 

\begin{corollary}\label{cor:flat2} For all $e<1/2$ and for all $t$ small enough that Theorem \ref{thm:flat2}
applies, any graphon $g \in \regular_{e,t}$ that is weakly equivalent 
to a symmetric bipodal graphon is group equivalent to a symmetric bipodal graphon. 
\end{corollary}

\section{Proof of Theorem \ref{thm:scallop}}

As in the last section, we will prove a slightly extended version 
of Theorem \ref{thm:scallop}: 

\begin{theorem} \label{thm:scallop2}
Let $n \ge 1$ be an integer. For every $e \in \left ( \frac{n}{n+1}, \frac{n+1}{n+2}\right )$, with 
corresponding minimal triangle density $t_0$ (depending on $e$), 
and for all
$\Delta t$ sufficiently small, the optimal graphon with edge/triangle densities $(e, t_0 + \Delta t)$
is unique and $n+2$-podal, with $(n,2)$ symmetry. Asymptotically, 
$\Delta \B = \B(e,t) - 
\B(e,t_0)$ scales as $\sqrt{\Delta t}$.
and the Lagrange multiplier 
$\beta$ scales as $1/\sqrt{\Delta t}$.
In the optimal graphon, the diagonal entries are all $\exp(-\Theta(\beta))$ and, except for the 
$(n+1, n+2)$ and $(n+2,n+1)$ entries, the off-diagonal entries are all $1 - \exp(-\Theta(\beta))$.
\end{theorem}

\begin{proof} 
The proof of Theorem \ref{thm:scallop2} (and therefore
Theorem \ref{thm:scallop}) follows the same script 
as the proof of Theorem \ref{thm:flat2}, namely 
\begin{enumerate}
\item Using the proximity of an entropy-maximizing graphon $g$ at 
$(e,t)$ to the unique entropy-maximizing graphon $g_0$ at $(e,t_0)$ 
to define 
approximate podes $I_1, \ldots, I_{n+3}$ where the columns with 
$x\in I_j$ with $j \le n+2$ 
are $L^2$-close to the corresponding columns of $g_0$, and where
the exceptional set $I_{n+3}$ is small. 
\item Using the Euler-Lagrange equations to show that all of the graphon values are exponentially 
close to 0 or 1, except on $I_{n+1} \times I_{n+2}$, $I_{n+2} \times I_{n+1}$, or when one of the
coordinates is in $I_{n+3}$. 
\item Showing that the only possible \Value-maximizing columns are small 
perturbations of the columns of $g_0$, thus allowing us to reassign the 
points of $I_{n+3}$ to the other podes. 
\item Bounding the variation in $g(x,y)$ in each rectangle $I_i \times I_j$ by a small multiple of the variation
in other rectangles. Combining estimates, this shows that the variation in each rectangle is bounded by 
a small multiple of itself, and must therefore be zero. 
\item Analyzing the finite-dimensional space of $(n+2)$-podal graphons near $g_0$ and determining that the
best one has $(n,2)$ symmetry. We then determine 
how $S$, $\beta$, and various entries of the optimal graphon scale with $\Delta t$. 
\item On the first pass, steps (1--5) only apply at values of $t$ for which the Lagrange 
multipliers $\alpha$ and $\beta$ are well-defined and finite. Extending the results to all sufficiently
small values of $\Delta t$ then follows from Theorem \ref{thm:no-singular}.
\end{enumerate}

There is one important difference between the situation of Theorem \ref{thm:flat2} and that of 
Theorem \ref{thm:scallop2}. The
additional podes that appear on the scallops provide an additional, and more efficient, means of generating
entropy at the expense of added triangles. As a result, $\Delta \B$ scales as $\sqrt{\Delta t}$
rather than $\Delta t \ln(1/\Delta t)$. Before getting into the details of the proof, we explain how
this works, starting near the first scallop, 
with $e \in (\frac12, \frac23)$. 

Consider tripodal graphons of the form shown in Figure \ref{fig:Tri01p}. The total edge density is 
\be e = 2c(1-c) + \frac12 p(1-c)^2, \ee
so we must have 
\be p = \frac{2(e-2c(1-c))}{(1-c)^2}. \label{eq:p} \ee
The triangle density is 
\be t = \frac32 pc(1-c)^2 = 3c(e-2c(1-c)) = 3ec-6c^2 + 6c^3. \ee
Taking derivatives, we see that 
\be \frac{dt}{dc} = 3(6c^2 - 4c+e) \hbox{ and } \frac{d^2t}{dc^2} = 36c-12. \ee
The first derivative is zero when 
\be c = \frac13 \left ( 1 + \sqrt{1 - \frac{3e}{2}} \right ). \ee
Since $d^2 t/dc^2$ is always positive, this gives the minimum triangle density among graphons of 
this kind. 
In fact, it minimizes $t$ among all possible graphons \cite{PR} and
is the unique optimal graphon with densities $(e, t_0)$ \cite{RS1}. 

\begin{figure}[ht]
\includegraphics[width=3in]{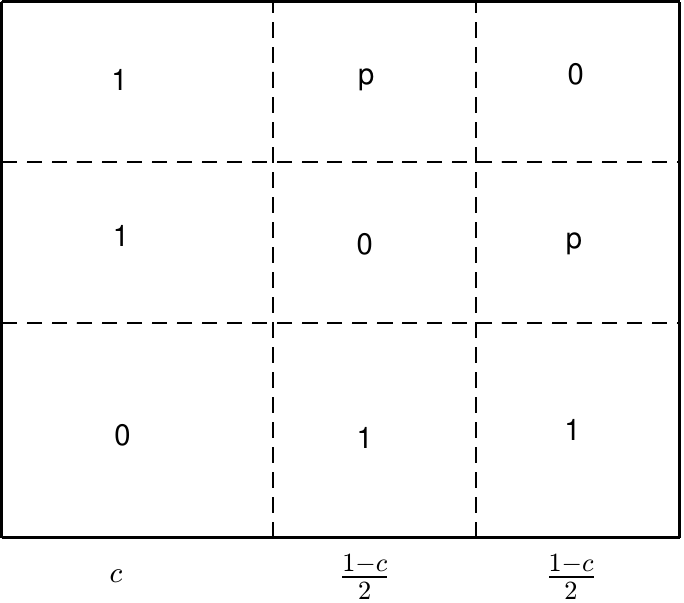}
\caption{A tripodal graphon of the form seen on the first scallop}\label{fig:Tri01p}
\end{figure} 

Now imagine varying $c$ and $p$ while preserving the structure of Figure \ref{fig:Tri01p}. 
The entropy of the graphon displayed in Figure \ref{fig:Tri01p} is
\be S = \frac12 (1-c)^2 H(p), \ee
where $p$ is given by equation (\ref{eq:p}). A little algebra then gives 
\be p = 4 - 4(1-c)^{-1} + 2e(1-c)^{-2}, \ee
so 
\be \frac{dp}{dc} = 4e(1-c)^{-3} - 4(1-c)^{-2} = \frac{4(e-(1-c))}{(1-c)^3}. \ee
We then compute 
\begin{eqnarray}
\frac{dS}{dc} & = & -(1-c)H(p) + \frac12 (1-c)^2 H'(p) \frac{dp}{dc} \cr 
& = & -(1-c)H(p) + \frac{2H'(p)(e+c-1)}{1-c} \cr 
& = & -(1-c) H(p) + \frac{H'(p)}{1-c}  (p(1-c)^2 + 6c-4c^2-2) \cr 
& = & (1-c)(pH'(p)-H(p)) + (4c-2) H'(p) \cr 
& = & (1-c) \ln(1-p) + 4c-2) (\ln(1-p)-\ln(p)) \cr 
& = & (3c-1) \ln(1-p) + 2(1-2c) \ln(p). 
\end{eqnarray}
Since $c$ is between $\frac12$ and $\frac23$, the coefficients of $\ln(1-p)$ and $\ln(p)$ are 
both positive, making $\frac{dS}{dc}$ negative.
We can increase the entropy to first order by decreasing $c$. That only increases the 
triangle count to second order in $\Delta c$, so we have achieved an entropy increase that scales 
as the square root of $\Delta t$.

The situation is similar near the other scallops. There is a family of graphons parametrized by the 
size $c$ of each of the $n$ identical podes, as in Figure \ref{fig:Multi01p} for $n=3$. 
There is a value $c_0$ that minimizes the triangle density, 
but $dS/dc$ is not zero at $c=c_0$. Instead, the calculation shown in the next paragraph shows that 
$dS/dc$ is negative for all relevant values of $c$. 
As a result, we can increase $S$ to first order in $\Delta c$ by decreasing $c$ 
while only increasing $t$ to second order, so $\Delta S \sim \sqrt{\Delta t}$. 

\begin{figure}[ht]
\includegraphics[width=4in]{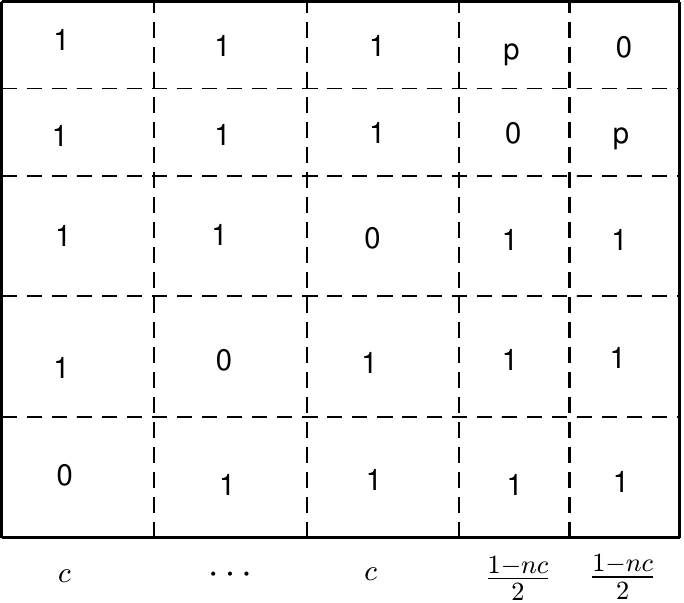}
\caption{A multipodal graphon of the form seen on the scallops, in this case with $n=3$}\label{fig:Multi01p}
\end{figure}

For a general value of $n$, the graphon on the $n$-th scallop is 1 everywhere except on the diagonal 
blocks and on the
two off-diagonal blocks in the upper right corner. The edge density is 
\be e  =  1- nc^2 - (1-nc)^2 + \frac{(1-nc)^2}{2} p. \ee 
This means that 
\begin{eqnarray}
p & = & \frac{2}{(1-nc)^2} \left ( e + nc^2 - 1 + (1-nc)^2 \right ) \cr 
& = & 2 \left ( e - \frac{n-1}{n} \right ) (1-nc)^{-2} - \frac4n (1-nc)^{-1} + \frac{2(n+1)}{n}.
\end{eqnarray}
Taking a derivative with respect to $c$ is then easy: 
\be \frac{dp}{dc} = \frac{4n}{(1-nc)^3} (e +c -1). \ee
The entropy is 
$S = \frac{(1-nc)^2}{2}H(p)$ and derivative of $S$ with respect to $c$ is  
\begin{eqnarray}
\frac{dS}{dc} & = & -n(1-nc)H(p) + \frac12 (1-nc)^2 H'(p) \frac{dp}{dc} \cr 
& = & -n(1-nc)H(p) + \frac{2nH'(p)(e+c-1)}{1-nc} \cr 
& = & -n(1-nc) H(p) + nH'(p)((n+1)c - 1 + p)  \cr 
& = & n(1-nc)(pH'(p)-H(p)) + 2n((n+1)c-1) H'(p) \cr 
& = & -n(1-nc) \ln(1-p) + 2n((n+1)c-1) (\ln(1-p)-\ln(p)) \cr 
& = & n((n+2)c-1) \ln(1-p) + 2n(1-(n+1)c) \ln(p). 
\end{eqnarray}
Since $c$ is between $\frac{1}{n+2}$ and $\frac{1}{n+1}$, the coefficients of $\ln(p)$ and 
$\ln(1-p)$ are both positive, making each term negative, so $\frac{ds}{dc} < 0$, as claimed.

\subsection{Defining approximate podes}
We now turn to the details of the proof. 
As usual, let $g_0$ be the unique entropy maximizer $g_0$ at $(e,t_0)$. 
This graphon takes the form shown in Figure \ref{fig:Multi01p}, with 
\be c = c_0 = \frac{1 + \sqrt{1- \frac{n+2}{n+1}e}}{n+2}, \ee
which is the value of $c$ that minimizes 
\be t = n(n+1)(n+2)c^3 - 3n(n+1)c^2 + 3nec. \ee
As we
approach the scallop, any sequence of entropy maximizers must converge (after applying measure preserving
transformations of $[0,1]$) to $g_0$ in $L^2$ by exactly the same
argument as in the proof of Theorem \ref{thm:flat2}. 

As before, any sequence of entropy maximizing reduced graphons must converge to $[g_0]$ in the cut metric, implying 
than any sequence of entropy maximizing graphons must are equivalent to graphons $g_i$ that converge to $g_0$ in the cut distance. 
By the definition of the cut distance, the average value of $g_i$ on each rectangle defined by the podes of $g_0$ must 
approach the (constant) value of $g_0$ on that rectangle. Since $\lim S(g_i)=S(g_0)$, the variance of $g_i$ must go to zero on 
each of these rectangles, so the graphons $g_i$ converge to $g_0$ in $L^2$. 

We pick a sufficiently small value of $\epsilon$
and consider values of $t$ small enough that $\|g-g_0\|_{L^2} < \epsilon$ for each optimal graphon $g$.
Let $I_1$, $I_2$, $\ldots$, $I_{n+2}$ be the subsets of the podes of $g_0$ for which $g_x$ lies
within $\sqrt{\epsilon}$ in $L^2$ of the corresponding column of $g_0$ and let 
$I_{n+3}$ be the exceptional set where $g_x$ is not close to the corresponding column of $g_0$. 
Note that $I_{n+3}$ may contain points $x$ where
$g_x$ is close to a different column of $g_0$. Those points will soon be reassigned.

\subsection{Variational equations} Next we need to compute $G(x,y)$ in different cases. Let 
$G_{i,j}$ denote a typical value of $G(x,y)$ when $x \in I_i$ and $y \in I_j$. 
We can estimate these quantities to within $O(\sqrt{\epsilon})$ using the 
columns of $g_0$.
Thanks to our 
$(n,2)$ symmetry, there are only five 
different numbers to compute, namely $G_{1,1}$, $G_{1,2}$, $G_{1,{n+1}}$, $G_{n+1,n+1}$ and 
$G_{n+1,n+2}$. The results are
\begin{eqnarray}
G_{1,1} & \approx & 1-c, \cr 
G_{1,2} & \approx & 1-2c, \cr 
G_{1,n+1} & \approx & (n-1)c + \frac{1-nc}{2} p, \cr 
G_{n+1,n+1} & \approx & nc + \frac{1-nc}{2} p^2, \cr 
G_{n+1,n+2} & \approx & nc,
\end{eqnarray}
where ``$\approx$'' means ``equal to within $O(\sqrt{\epsilon})$''. 
Note that $G_{1,1}$ and $G_{n+1, n+1}$ are greater than $G_{n+1, n+2}$ by amounts that are 
$\Omega(1)$ as $t \to 0$ while $G_{1,2}$, and $G_{1,n+1}$ are less than $G_{n+1,n+2}$ by amounts 
that are $\Omega(1)$. Multiplying by $\beta$ and adding $\alpha$, and using the fact that 
\be H'(p) = \alpha + \beta G_{n+1, n+2} = (\alpha + nc) + O(\beta \sqrt{\alpha}). \ee
we get that $H'(g(x,y))$ is $\Omega(\beta)$ on the diagonal rectangles that do not involve $I_{n+3}$
and is $-\Omega(\beta)$ on the off-diagonal blocks that do not involve $I_{n+3}$, with the exception
of $I_{n+1}\times I_{n+2}$ and $I_{n+2}\times I_{n+1}$. This implies that $g(x,y)$ is exponentially
small (that is, $\exp(-\Omega(\beta))$) on the diagonal blocks, exponentially close to 1 on all but two of the off-diagonal blocks, 
and of course is close to $p$ on $I_{n+1} \times I_{n_2}$ and $I_{n+2}\times I_{n+1}$.

\subsection{Maximizing \Value}

Let $C: [0,1] \to [0,1]$ be a function whose \Value{} we aim to 
maximize. For each $i=1,2,\ldots, n+2$,
let $a_{n+2}$ be the average values of $C(y)$ for $y \in I_i$. Of the terms contributing to 
$W(C)$, the entropy term is bounded above by $2c \sum_{i=1}^n H(a_i) + (1-nc) (H(a_{n+1}) 
+ H(a_{n+2}))$, since fluctuations in $C$ within each pode can only decrease the entropy. 
The edge density term is $-\alpha (2c \sum_{i=1}^n a_i + (1-nc) (a_{n+1}+a_{n+2}))$. 

The most important term comes from triangles. To within the accuracy of our approximation
that $g(y,z)$ is constant on each rectangle, it is the quadratic function 
\be - \beta \sum_{i,j=1}^{n+2} M_{ij} a_i a_j, \ee
where $M_{ij}$ is the integral of $g(y,z)$ over $I_{i} \times I_j$.

If we are at a maximum of $W$, then the gradient of $W$ must be zero and the Hessian must be 
negative semi-definite. The Hessian of $W$ with respect to the variables $\{a_i\}$ 
is precisely $-2\beta M$ plus diagonal 
terms proportional to $H''(a_i)$.
The matrix $M$ is (nearly) zero on the diagonal, with all of the off-diagonal terms being close to 1
or $p$, and so has eigenvalues of both signs. The only way for the Hessian to be negative semi-definite 
is for all but one of entries $H''(a_i)$ to be at least of order $\beta$. In other words, all columns 
that maximize $W$ must have every entry but one (or every entry) 
approximately equal to 0 or 1. In terms 
of $G$, for $y$ in any pode but one, $|G(x,y) -nc|$ must exceed 
$\Theta(1/\beta)$. 

We now examine the possibilities. 
\begin{itemize}
\item If $G(x,y) \approx nc$ for $y\in I_1$, then $2c\sum_{i=2}^n a_i + (1-nc) (a_{n+1}+a+{n+2})
\approx 2nc$. But that is impossible if each $a_i$ (other than $a_1$) is equal to 0 or 1. 
Contradiction. Likewise, it is not possible to have $G(x,y) \approx nc$ for $y \in I_2, \ldots, I_n$.
The first $n$ variables $a_i$ are all either pegged to 0 or to 1. 
\item If two or more of the variables $a_1, \ldots, a_n$ are pegged to 0, then $G(x,y) < nc$ for 
all $y$, so $g(x,y) \approx 1$ for all $y$, which is a contradiction. Thus either one or none 
of the first $n$ $a_i$'s is pegged to 0 and the rest are pegged to 1. 
\item If exactly one of these variables is pegged to 0, then for $y\in I_{n+1}$ or $y\in I_{n+2}$ 
we have $G(x,y) \le p \frac{1-c}{2} + (n-1) c < nc$, so $a_{n+1}$ and $a_{n+2}$ are pegged to 1. 
In other words, our column is just like the columns when $x$ is in one of the first $n$ podes.  
\item If all of the variables $a_1, \ldots, a_n$ are pegged to 1, then we examine $a_{n+1}$ and 
$a_{n+2}$. Neither one is pegged to 1, since $G(x,y)$ is at least $nc$ for $y$ in either $I_{n+1}$ or 
$I_{n+2}$. They cannot both be pegged to 0, since that would make $G(x,y)=nc$ in both podes, meaning
that the values are {\em not} pegged and the Hessian is not negative-definite. 
Thus one value must be pegged to 0 while the other is intermediate between 0 and 1. 
The closeness of all but one $a_i$ to 0 or 1 gives us the same equation
for the remaining $a_i$ as satisfied by the actual columns of $I_{n+1}$
or $I_{n+1}$, implying that the final $a_i$ must be close to $p$. That
is, $C$ is close to the actual columns when $x \in I_{n+1}$ or $I_{n+2}$.
\end{itemize}

The upshot is that all \Value-maximizers are already $L^2$-close 
to columns
of $g_0$. Since each column is a \Value-maximizer, we can reassign 
all of the points of $I_{n+3}$ to other podes. Note that this reassignment
can result in the podes of $g$ having sizes that are slightly 
different from those of the
corresponding podes of $g_0$. 

Controlling the columns to within a small $L^2$ error gives us pointwise
control over $G(x,y) = \langle g_x | g_y \rangle$. 
On all rectangles except for $I_{n+1}\times 
I_{n+2}$ and $I_{n+2} \times I_{n+1}$, this forces $g(x,y)$ to be 
exponentially close to 0 or 1. This makes $G(x,y)$ exponentially close
to constant on $I_{n+1} \times I_{n+2}$ and $I_{n+2} \times I_{n+1}$
and so makes $g(x,y)$ exponentially close to a constant (that is 
close to $p$, but not necessarily exponentially close) on these
rectangles.

\subsection{Exact multipodality}

So far we have shown that an optimal graphon has to be approximately multipodal. There 
are $n$ podes $I_1, \ldots, I_n$ of width close to 
\be c = \frac{1 + \sqrt{1 - \frac{n+2}{n+1} e}}{n+2} \ee
and two podes of width close to $\frac{1-nc}{2}$. The graphon is exponentially close to 0 on the diagonal
blocks, exponentially close to 1 on all of the off-diagonal blocks but two, and close to $p$ on 
$I_{n+1}\times I_{n+2}$ and $I_{n+2}\times I_{n+1}$. We next show that the graphon is exactly constant
on each rectangle. The proof is essentially a rerun of the analogous step in the proof of Theorem
\ref{thm:flat2}, only with more terms. 

Let $g_{ij}$ be the average value of the graphon on the rectangle $I_i \times I_j$ and let 
$\Delta g_{ij}$ be the difference between the greatest and lowest value of the graphon in that rectangle. Let $c_i$ be 
the width of $I_i$. We have already determined that all $g_{ij}$'s except for $g_{n+1,n+2}$ and 
$g_{n+2,n+1}$ are exponentially close (in $\beta$) to 0 or 1, and hence that $1/H''(g_{ij})$ is 
exponentially small. 

If $x \in I_i$ and $y \in I_j$, then the maximum and minimum possible values of $G(x,y)$, and their
difference, are 
\begin{eqnarray}
\hbox{max} & = & \sum_{k=1}^{n+2} c_k (g_{ik}+\Delta g_{ik})(g_{jk} + \Delta g_{jk}), \cr 
\hbox{min} & = & \sum_{k=1}^{n+2} c_k (g_{ik}-\Delta g_{ik})(g_{jk} - \Delta g_{jk}), \cr 
\hbox{difference} & = & \sum_{k=1}^{n+2} 2c_k (g_{ik} \Delta g_{jk} + g_{jk} \Delta g_{ik}).
\end{eqnarray} 
Applying the mean value theorem to the Euler-Lagrange equations, 
and noting that $H''(u)$ is always negative, we have 
\be -H''(g_{ij,0}) \Delta g_{ij} \le 2 \beta \sum_{k=1}^{n+2} c_k (g_{ik} \Delta g_{jk} + g_{jk} \Delta g_{ik}), \ee
where $g_{ij,0}$ is some number between the values of $g$ corresponding to the maximum and minimum possible
values of $G(x,y)$. 

The sum on the right contains terms proportional to $\Delta g_{ij}$ itself, coming from $k=i$ or $k=j$. We bring those terms to the left hand side, noting that coefficients of those terms are much smaller than $H''(g_{ij,0})$.
When $\{i,j\} \ne \{n+1,n+2\}$, this is because $g_{ij,0}$ is exponentially close to 0 or 1, so $H''(g_{ij,0})$ is exponentially large, while the coefficients on the right hand side are $O(\beta)$. 
When $\{i,j\} = \{n+1,n+2\}$, this 
is because the coefficients of $\Delta g_{n+1,n+2}$ on the right hand side are proportional to 
$\beta g_{n+1,n+1}$ or $\beta g_{n+2,n+2}$, both of which are exponentially small. By changing the 
factor of 2 on the right hand side to a 3, we can absorb these small corrections to the coefficient of 
$\Delta g_{ij}$ and also replace $H''(g_{ij,0})$ with just $H''(g_{ij})$. We also bound $g_{ik}$ and 
$g_{jk}$ by 1. The upshot is that 
\begin{eqnarray}
\Delta g_{ij} & \le & \frac{3\beta}{-H''(g_{ij})} \sum_k c_k (\Delta g_{jk} + \Delta g_{ik}) \cr 
& \le & \frac{6\beta}{-H''(g_{ij})} \max(\Delta g_{ik} \hbox{ or } \Delta g_{jk}),
\end{eqnarray}
where the sum on the first line and the maximum on the second line skips terms involving $\Delta g_{ij}$
itself. 

Whenever $\{i,j\} \ne \{n+1,n+2\}$, $\frac{6\beta}{-H''(g_{ij})}$ is exponentially small, so $\Delta g_{ij}$
is bounded by a tiny multiple of a sum of similar errors. In particular, the largest $\Delta g_{ij}$ of 
this sort is bounded by a sum of contributions much smaller than itself, possibly plus a contribution
from $\Delta g_{n+1, n+2}$. The conclusion is that all $\Delta g_{ij}$'s other than $\Delta g_{n+1,n+2}$ 
are bounded by $\beta \exp(-\Omega(\beta)) \Delta g_{n+1,n+2}$. 

Now consider the equation for $\Delta g_{n+1,n+2}$. This equation 
indicates that $\Delta g_{n+1,n+2}$ is bounded by an $O(1)$ multiple of $\beta$ times the largest of 
the remaining $\Delta g_{ij}$'s, and so is bounded by a constant times 
$\beta^2 \exp(-\Omega(\beta)) \Delta g_{n+1,n+2}$. When $\beta$ is large, $\Delta_{n+1,n+2}$ is thus 
bounded by a constant
(less than one) times itself, and so must be zero. But then all of the other $\Delta g_{ij}$'s must also
be zero, so our graphon is multipodal.

\subsection{Graphons with $(n,2)$ symmetry}

Finally, we show that that the optimal graphon is symmetric in the first $n$ podes and symmetric in the 
last two. Let $c_1, \ldots, c_{n+2}$ be the sizes of the various podes, let $\bar c$ be the average 
size of the first $n$ podes, and let $\Delta c_i$ be $c_i-\bar c$ or $c_i - \frac{1-n\bar c}{2}$, depending
on whether we are talking about the first $n$ podes or the last two. Let 
$W_i$ be the \Value{} of columns in the $i$-th pode. There are five kinds of rectangles, namely 
$I_i \times I_j$ with $i=j \le n$, with $i < j \le n$ or $j<i \le n$, with $i \le n < j$ or 
$j \le n < i$, with $i=j >n$, and finally with $\{i,j\} = \{n+1,n+2\}$. In each class, let $\bar g_{ij}$
be the average value of the graphon and let $\Delta g_{ij} = g_{ij} - \bar g_{ij}$. We also refer to 
$g_{n+1,n+2}$ as $p$.

A key fact is that all of the entries in the first $n$ columns are exponentially close to 0 or 1. Meanwhile,
the Euler-Lagrange equations for $g_{n+1,n+2}$ say that 
\begin{eqnarray} 
H'(p) & \approx & \alpha + \beta \sum_{i=1}^n c_i,  \cr 
\alpha & \approx  & H'(p) - \beta \sum_{i=1}^n c_i,
\end{eqnarray}
where ``$\approx$'' means ``equal up to exponentially small corrections''. 

Now suppose that $i$ and $j$ are indices less than or equal to $n$. Since all of the entries $g_{ik}$ 
and $g_{jk}$ are exponentially close to 0 or 1, the entropy contribution to $W_i$ or $W_j$ is 
exponentially small. The coefficient of $\alpha$ is $\sum_{k \ne i} c_k = 1-c_i$, while the coefficient
of $\beta/2$ is the integral of the graphon over everything that doesn't involve the $i$-th pode. The
upshot is that 
\begin{eqnarray} 
W_i & \approx & -\alpha(1-c_i) - \frac{\beta}{2} (e - 2 c_i(1-c_i)) \cr 
& = & - \left ( \alpha + \frac{\beta e}{2} \right ) + c_i(\alpha + \beta) - c_i^2 \beta,
\end{eqnarray}
with a similar result for $W_j$. Taking the difference gives 
\begin{eqnarray} 
0 & = & W_i - W_j \cr 
& \approx & (c_i-c_j) (\alpha + \beta - \beta(c_i+c_j)) \cr 
& \approx & (c_i-c_j) (H'(p) + \beta(c_{n+1}+c_{n+2}-c_i-c_j)),
\end{eqnarray}
where we have used the fact that $\sum_{i=1}^n c_i = 1-c_{n+1}-c_{n+2}$. 
However, $c_{n+1}$ and $c_{n+2}$ are close to $\frac{1-n\bar c}{2}$, while $c_i$ and $c_j$ are close
to $\bar c$, so the coefficient of $\beta$ is bounded away from zero. We conclude 
that $c_i-c_j$ must be exponentially small. More precisely, $c_i - c_j$ must be exponentially smaller than
the largest $|\Delta g_{ik}|$ or $|\Delta g_{jk}|$. A similar argument shows that $c_{n+1}-c_{n+2}$ is
also exponentially smaller than the largest $|\Delta g|$.

We now look at the Euler-Lagrange equations for $g_{i\ell}$ and $g_{j\ell}$, 
where $\ell$ is different from $i$ or $j$. The difference between $G(x,y)$ in $I_i \times I_\ell$ and 
$I_j \times I_\ell$ is 
\begin{eqnarray} 
\sum_{k=1}^{n+2} c_k g_{\ell k} (g_{ik}-g_{jk}) & =&    \sum_{k=1}^{n+2} c_k g_{\ell k} (\Delta g_{ik}-\Delta g_{jk}) \cr 
&& + (c_i g_{\ell i} - c_j g_{\ell j})(\bar g_{ii} - \bar g_{ij}).
\end{eqnarray} 
The first line is of the order of the largest $\Delta g$. The second line has a similar contribution
from the difference of $\Delta g_{\ell i}$ and $\Delta g_{\ell j}$, plus a contribution of order 
$c_i - c_j$. But then 
\be H'(g_{i\ell}) - H'(g_{j\ell}) = \beta (G_{i\ell} - G_{j\ell}),\ee
which is $\beta$ times a linear combination of $\Delta g$'s and $c_i-c_j$. Since $H''$ is enormous on
the interval from $g_{i\ell}$ to $g_{j\ell}$ (both of which are exponentially close to 1), 
$\Delta g_{i\ell}-\Delta g_{j\ell}$ is bounded by a tiny combination of other $\Delta g$'s and $\Delta c$'s. 

Repeating this argument for $g_{ii} - g_{jj}$ and for $g_{n+1,n+1}-g_{n+2,n+2}$, we get that 
\begin{itemize}
\item The biggest $\Delta c_i$ is bounded by a tiny constant times the biggest $\Delta g$. 
\item The biggest $\Delta g$ is bounded by a tiny constant times the biggest $\Delta c$. 
\end{itemize}
We conclude that all of the $\Delta c$'s and $\Delta g$'s are zero. 

We have determined the form of the optimal graphon $g$ at $(e,t)$. 
Noting that the Boltzmann entropy $\B(e,t)$ equals the Shannon entropy 
$S(g)$ of the optimal graphon $g$, all statements about $S$ are easily converted into statements 
about $\B$. 

Finally, we must show that the values of $g$ on each rectangle, and 
the sizes of the different podes, are analytic functions of $(e,t)$. 
This follows from a general principle in algebraic geometry, which in
turn is essentially just the implicit function theorem. Within the
product of the Razborov triangle and the finite-dimensional space 
of graphons with $(n,2)$ symmetry, the set of optimal graphons is a 
2-dimensional analytic variety, cut out by the analytic Euler-Lagrange
equations. As long as the 
tangent space does not degenerate, we can write all but two of the 
variables as analytic functions of the last two, which we can
choose to be $(e,t)$. 

\end{proof}

As with Theorem \ref{thm:flat2}, we have actually proven that an optimal graphon
in $\regular_{e,t}$ has a unique form up to group equivalence, implying

\begin{corollary} \label{cor:scallop2} For all $(e,t)$ such that Theorem 
\ref{thm:scallop2} applies, any graphon $g \in \regular_{e,t}$ that is weakly equivalent to the 
multipodal entropy maximizer described in that theorem is actually group 
equivalent to the multipodal entropy maximizer. 
\end{corollary}

\subsection{Distinct phases and rank}\label{OrderParam}

\begin{theorem} \label{thm:orderparameter2}
Each of the phases above the scallops proven in Theorems \ref{thm:flat2} and \ref{thm:scallop2} have unique optimal graphons with distinct symmetries and cannot be analytically continued to one another.
\end{theorem}

\begin{proof} 
The optimal graphons described by Theorem 
\ref{thm:flat2} have rank 2, while the optimal graphons above the 
$n$-th scallop described by Theorem \ref{thm:scallop2} have rank $n+2$.
We will construct a sequence of ``order parameters,'' each a polynomial
in finitely many subgraph densities, to distinguish between graphons 
of different rank. Specifically, the $k$th order parameter is 
identically zero whenever the rank of the optimal 
graphon is $k-1$ or less, and is never zero when the rank of the graphon is 
$k$. Since an analytic function on a connected set 
that is zero on an open subset is zero everywhere, there cannot be 
an analytic path connecting the $(k-2)$-nd scallop (where the graphon
has rank $k$ and the order parameter is nonzero) to the previous scallops
or to the $A(2,0)$ phase, where the order parameter is zero. 
In other words,
the phases above the different scallops are all distinct. 

Newton's identities relate the determinant of a $k \times k$ matrix $A$ to
the traces of $A^j$ for $j=1,2,\ldots, k$. For instance, if we let $t_j=\TR(A^j)$, then the determinants of small matrices are given by 
the formulas
\be \det(A) = \begin{cases} (t_1^2-t_2)/2 & k=2, \cr 
(t_1^3-3t_1t_2+2 t_3)/6 & k=3, \cr 
(t_1^4 - 6 t_1^2 t_2 + 8 t_1 t_3 -3 t_4)/24 & k=4. \end{cases} \ee
Let $p_k(A)$ be the polynomial in the variables $\{t_j\}$ 
that gives the determinant of a $k\times k$ matrix $A$. 

The same ideas work for arbitrary diagonalizable linear operators,
for which the rank equals the number of nonzero eigenvalues, counted
with multiplicity.
If we evaluate $p_k$ on any diagonalizable trace-class
operator,
we get zero if the rank of the operator is less than $k$ and the product of the
nonzero eigenvalues (counted with multiplicity) 
if the rank is equal to $k$. The key algebraic fact is that,
for operators of rank $k$ or less, we have 
\be t_j = \lambda_1^j + \cdots + \lambda_k^j, \ee
where some of the eigenvalues $\lambda_i$'s may be zero, 
and $p_k$ computes
$\lambda_1 \cdots \lambda_k$, which is nonzero precisely when there are 
$k$ nonzero eigenvalues (counted with multiplicity).  

In particular, we can 
apply these formulas to graphons. (Graphons are always diagonalizable, being symmetric and trace class.) 
For instance, the expression 
$(t_1^4 - 6 t_1^2 t_2 + 8 t_1 t_3 -3t_4)/24$, where now $t_j=\TR(g^j)$, 
gives zero if the rank of the graphon $g$ is less than 4 and gives a nonzero
number if the rank is equal to 4. 

When $j>2$, $t_j$ is the density of $j$-gons. The problem is that we cannot 
realize $t_1$ and $t_2$ as subgraph densities, so we cannot assume {\em a priori} that
$t_1$ and $t_2$ are analytic functions of $(e,t)$ in each phase. 
To get around this problem, we define our $k$th order parameter to be 
$p_k(g^3)$. This is still a polynomial in $\{t_j\}$, only now $j$ 
ranges from $3$ to $3k$ in steps of 3. 
In particular, $t_j$ is the density of $j$-gons for each applicable $j$. 

The $k$-th order parameter is then zero if $g^3$ has rank less than 
$k$ and is nonzero if $g^3$ has rank $k$. 
But $g^3$ has the same rank as $g$, so we are actually testing the 
rank of $g$. 

In summary: the $k$th order parameter is an analytic function of 
$(e,t)$ in each phase, being built from subgraph densities. It is identically
zero on the regions above the 0th, 1st, $\ldots$, $(k-3)$rd scallops but is never
zero on the region above the $(k-2)$nd scallop. Thus the region above the $(k-2)$nd
scallop is in a different phase from the regions above all the previous 
scallops. Each scallop has its own unique phase. 

\end{proof}

\section{Proof of Theorem \ref{thm:top}}
Once again we prove a slightly stronger version of the theorem
stated in the introduction. 

\begin{theorem}\label{thm:top2}
For each fixed $e \in (0,1)$ and all $t$ sufficiently close to (but below) $e^{3/2}$, the 
optimal graphon with edge/triangle densities $(e,t)$ is unique and bipodal. Asymptotically, the Boltzmann entropy 
scales as $-(e^{3/2}-t)\ln(e^{3/2}-t)$
and the Lagrange multiplier $\beta$ scales as $\ln(e^{3/2}-t)$. 
\end{theorem}

\begin{proof}
We follow the same overall roadmap as 
the proofs of Theorems \ref{thm:flat2} and \ref{thm:scallop2}. 
Specifically, 
\begin{enumerate}
\item Using the proximity to the upper boundary, we break $[0,1]$
into two large podes $I_1$ and $I_2$ and a small exceptional set 
$I_3$ such 
that $g_x$ is $L^2$-close to the indicator function of $I_1$ when
$x \in I_1$ and is $L^2$-close to zero when $x \in I_2$. 
\item Equating the \Values{} of $g_x$ when $x \in I_1$ to those of 
$g_x$ when $x \in I_2$, we determine that $\beta/\alpha
\approx -2/e$. The multiplier $\beta$ is large and negative,
while $\alpha$ is large and positive. 
\item Maximizing $W(C)$ for an arbitrary $C: [0,1] \to [0,1]$, we show 
that every column is close to a typical column in the first or second
pode. After reassigning points,
$I_3$ is then empty. The control this gives us on $G(x,y)$ shows that 
$g(x,y)$ is everywhere exponentially close to 0 or 1.
\item Bounding the fluctuations in each rectangle by multiples 
of the fluctuations in other rectangles to show that all 
fluctuations are in fact zero. In other words, our optimal graphon is 
exactly bipodal with values that are exponentially close to 0 or 1. 
\item Using Theorem \ref{thm:no-singular} to eliminate the possibility
that some points $(e,t)$ might have entropy maximizers without well-defined Lagrange multipliers 
$(\alpha, \beta)$.
\end{enumerate}

Step 1 is identical to what we have done before. There is a 
unique reduced graphon at $(e, e^{3/2})$,
namely the equivalence class of a graphon $g_0$ that is 1 on $I_1 \times I_1$ and zero elsewhere, where $I_1$ is a pode of size $\sqrt{e}$. Every graphon with $t$ 
close to $e^{3/2}$, and in particular any entropy-maximizing graphon,
must be $L^2$ close to $g_0$. This means that for all $x$'s outside
of a set of small measure, $g_x$ is $L^2$-close to the corresponding
column of $g_0$. This also implies that $G(x,y)$ is close to $\sqrt{e}$
when $x$ and $y$ are both in $I_1$ and is close to zero when either
is in $I_2$. 

The \Value{} of a column that is nearly zero is of course nearly zero.
The \Value{} of a column that is nearly 1 on $I_1$ and nearly zero
elsewhere is approximately 
\be - 2\alpha \sqrt{e} - \beta e^{3/2}. \ee
Since all columns must have the same \Value, we must have 
$\beta/\alpha \approx -2/e$. The Lagrange multipliers $\alpha$
and $\beta$ diverge at the same rate as we approach the boundary,
with $\alpha \to \infty$ and $\beta \to -\infty$. 

Now consider an arbitrary function $C: [0,1] \to [0,1]$. 
Let $a$ be the 
average of $C(y)$ on $I_1$ and let $b$ be the average on $I_2$. 
Using the approximation that $G$ is $L^2$-close to $\sqrt{e}$ 
times the indicator function of $I_1 \times I_1$, we get that 
\be W(C) \le 2\sqrt{e} H(a) + 2(1-\sqrt{e}) H(b) - 2\alpha \sqrt{e} a
- 2\alpha (1-\sqrt{e}) b - \beta e^{3/2} a^2, \ee
with equality if $C$ is constant on $I_1$ and constant on $I_2$. 
Since $\alpha$ is large and positive, we must have $b$ exponentially
close to 0. Setting $b \approx 0$, our \Value{} is then approximately
\be \alpha \sqrt{e} (a^2 - a). \ee
This is of course maximized at the endpoints $a=1$ and $a=0$, 
being negative when $a \in (0,1)$. In other words, 
any \Value-maximizing column must either have $a \approx 1$ and 
$b \approx 0$, and so must be close to the columns in $I_1$, or 
$a \approx 0$ and $b \approx 0$, and so must be close to the 
columns in $I_2$. Reassigning the points of $I_3$ to $I_1$ or $I_2$
accordingly, we obtain a situation where $I_3$ is empty. 

To constrain the fluctuations in $g(x,y)$ in each rectangle, we
recall the variational equations 
\be H'(g(x,y)) = \alpha + \beta G(x,y). \ee
Since $G(x,y) \approx 0$ or $\sqrt{e}$, depending on which 
quadrant we are in, this implies that $g(x,y)$ is exponentially 
close to 1 on $I_1 \times I_1$ and exponentially close to 
0 on $I_1 \times I_2$ and $I_2 \times I_2$. In particular, $H''(g)$,
which scales as the larger of $1/g$ and $1/(1-g)$, is much larger
than $|\beta|$. Looking at the change in the left hand side and 
right hand side of this equation within a single quadrant, we 
see that $H''(g)$ times the maximum fluctuation within any quadrant
is of the same order as $|\beta|$ times the maximum fluctuation 
within any quadrant. But that means that the maximum fluctuation is 
bounded by a small multiple of itself, and so must be zero. Our 
graphon is exactly bipodal. 

Finally, we do some calculations in the space of bipodal graphons. 
Let $g_{11}$, $g_{12}$ and $g_{22}$ be the values of the
optimal graphon on $I_1 \times I_1$, $I_1 \times I_2$ and $I_2 \times I_2$,
respectively. We treat $g_{11}$, $g_{12}$ and $g_{22}$ as free variables
and adjust the width of $I_1$ to keep the edge density fixed. To 
leading order, $e^{3/2}-t$ is a linear function of $(1-g_{11}, g_{12},
g_{22})$. However, $1-g_{11}$, $g_{12}$ and $g_{22}$ all scale as 
exponents of $\beta$, so $\beta$ must scale as $\ln(e^{3/2}-t)$. 
The entropy goes as $-(1-g_{11})\ln(1-g_{11}) - g_{12} \ln(g_{12})
- g_{22} \ln(g_{22})$, which then scales as $-(e^{3/2}-t)) \ln(e^{3/2}-t)$.

The analyticity of $g$ as a function of $(e,t)$ follows from the 
same argument as in the proof of Theorem \ref{thm:scallop2}, only with
the space of $(n,2)$ symmetric graphons replaced by the space of
bipodal graphons. 

\end{proof}

Finally, as with Theorems \ref{thm:flat2} and \ref{thm:scallop2}, we note that
we have determined the entropy maximizer up to group equivalence, implying 

\begin{corollary} \label{cor:top2} Let $(e,t)$ be such that Theorem 
\ref{thm:top2} applies. Then any graphon $g \in \regular_{e,t}$ that is weakly
equivalent to the bipodal maximizer defined in the theorem is actually group
equivalent to that bipodal maximizer. 
\end{corollary}

\section{ERGM-invisibility}\label{sec:No-ERGM} 

Using Lagrange multipliers is superficially similar to studying an 
exponential random graph model (ERGM; see Section 6 in \cite{Ch16}, or \cite{DL}, or \cite{CD} for a relevant introduction),
where one considers an ensemble of {\em all} graphs on $n$ vertices, where the probability of a given graph $G$
is proportional to 
\be \exp \left (-n^2 \left [ \alpha \edens(G) + \frac{\beta}{3} \tdens(G) \right ] \right ). \ee
Here $\edens(G)$ and $\tdens(G)$ are the edge and triangle densities of the graph $G$ and $\alpha$ and $\beta$ are variables which can move the distribution, expected to be narrowly peaked for large $n$. 
(In the literature, ERGMs are usually described in terms of parameters
$\beta_1=-\alpha$ and $\beta_2=-\beta/3$, but that linear change of variables does not matter.)

ERGMs are widely used to model real-world networks, with $n$ necessarily finite. See \cite{F1}, \cite{F2} and their bibliographies for references relevant to the current discussion.
By 2010 there was literature noting that fitting of parameters to data had various problems, and in \cite{CD} the recent LDP for $\G(n,p)$ graphs \cite{CV} was applied to see if that would help. See \cite{Bhamidi} and the introduction in \cite{CD} for background.
\cite{CD} illuminated some issues but left some others unresolved. Our treatment of the Boltzmann entropy can help, as follows.

The $n \to \infty$ limit of an ERGM with parameters $(\alpha, \beta)$ can be understood in terms of graphons and the 
function 
\be \Psi(\alpha, \beta) = \max_{g} \left ( S(g) - \alpha \edens(g) -\frac{\beta}{3} \tdens(g)\right ) \equiv \max_{g} F(g). \label{eq:psi} \ee

The LDP relates the graphon that maximizes $F(g)$ on the right hand side of (\ref{eq:psi}), for
the given values of $\alpha$ and $\beta$, to typical large graphs in the ensemble.
Note that $\Psi(\alpha, \beta)$ is the Legendre transform of the Boltzmann entropy $\B(e,t)$: 
\be \Psi(\alpha, \beta) = \max_{e,t} \left ( \B(e,t) - \alpha e - \frac{\beta}{3} t \right ). \label{eq:Legendre} \ee

If the Boltzmann entropy function were convex and the Razborov triangle were convex, then the Legendre transform 
(\ref{eq:Legendre}) would be invertible within each phase \cite{Tou}. If that were true, we could tune $\alpha$ and $\beta$ 
within each phase to get whatever 
values of $(e,t)$ we wanted. In statistical mechanics, this ability to switch back and forth
between fundamental variables and conjugate variables is called
{\em equivalence of ensembles}.  

With graphs, the Boltzmann entropy function is not convex and neither is
the Razborov triangle. There is no equivalence of ensembles; the Legendre transform (\ref{eq:Legendre}) is not invertible. 
Specifically, there are many values of $(e,t)$ for
which there do not exist {\bf any} values of $(\alpha, \beta)$ whose $F$-maximizing 
graphons have 
edge/triangle densities $(e,t)$. We call such points $(e,t)$ {\em ERGM-invisible}. 
We can still understand
$\Psi(\alpha, \beta)$ and the phases of an ERGM by studying $\B(e,t)$, since $\Psi$ is still the Legendre
transform of $\B$. However, we cannot understand $\B(e,t)$, or graphs with general densities $(e,t)$,
by studying $\Psi(\alpha, \beta)$. 

It has long been known that all points with $t$ greater than $e^3$, and even moderately
less than $e^3$, are ERGM-invisible \cite{CD}. 
The only points off the \ER{} curve that {\em might} be ERGM-visible 
lie close to the lower boundary of the  
triangle (Figure \ref{fig:phase_space}). We now show that, because of the nonconvexity of the Razborov triangle, most of those are also
ERGM-invisible. 

\begin{theorem} \label{thm:noERGM1} If $n$ is a positive integer and $\frac{n}{n+1} < e < \frac{n+1}{n+2}$, and if 
$t$ is sufficiently close to the minimum triangle density $t_0$, then $(e, t)$ is ERGM-invisible.
\label{thm:no-ERGM}
\end{theorem}

\begin{proof}
Fix a value of $e$ strictly between $\frac{n}{n+1}$ and $\frac{n+1}{n+2}$. The points with edge 
density $e$ and triangle density just above the minimum must have large values of the Lagrange 
multipliers $\alpha$ and $\beta$. However, the scallop itself is concave down, so for positive values of $\beta$, the 
linear function $\alpha e + \beta t$ is greater at one or both of the neighboring cusps 
(at edge density $\frac{n}{n+1}$ and $\frac{n+1}{n+2}$) than near the interior of the scallop. For large positive 
values of $\beta$ (and correspondingly large negative values of $\alpha$), this difference is greater than the 
bounded difference in Shannon entropy between the graphons described by Theorem \ref{thm:scallop2} and the 
zero-entropy graphons at the cusps. Since large values of $\beta$ correspond to small values of 
$\Delta t$, we conclude that, for all sufficiently small values of $\Delta t$, the point 
$(e, t_0 + \Delta t)$ is ERGM-invisible. 
\end{proof}

A similar result applies at the top of the Razborov triangle. 
\begin{theorem} \label{thm:noERGM2} 
For each $e \in (0,1)$ and for all $t$ sufficiently close to (and less than) $e^{3/2}$,
$(e,t)$ is ERGM-invisible. 
\end{theorem}

\begin{proof}
Near the top boundary,
$\alpha$ is large and positive while $\beta$ is large and negative. 
However, the boundary curve $t=e^{3/2}$ is concave up, so we can 
decrease $\alpha e + \beta t/3$ by moving to one endpoint $(0,0)$ 
or the other $(1,1)$. Whenever $\alpha$ and $\beta$ are big enough
in magnitude,
in other words whenever we are close enough to the top boundary, 
these gains swamp any changes in $S(g)$. Either
the constant graphon $g=0$ or the constant graphon $g=1$ yields a larger value of $F$ than the optimal graphon at
$(e,t)$. 
\end{proof}

The statement of Theorem \ref{thm:noERGM2} is not new; see \cite[Theorem 6.2]{CD}. However, the simplicity of the proof gets to the heart
of why points in this region are ERGM-invisible. 

Considering the Razborov triangle as a whole, only a few pieces are ERGM-visible. The \ER{} curve $t=e^3$ is 
ERGM-visible. Since $d\edens$ and $d\tdens$ are collinear at constant graphons, each point on the \ER{} 
curve actually corresponds to an infinite set of $(\alpha,\beta)$ values. A neighborhood of each cusp is ERGM-visible; that's
what you get when $\beta$ is large and positive. As far as we can tell, a neighborhood of the flat part of the lower
boundary is also ERGM-visible. In short, ERGMs can see homogeneous \ER{} or $A(n,0)$ structures, where every vertex looks like every other vertex. But that's all. 

In this paper, we have shown how a homogeneous constraint on the total 
number of edges and triangles leads to inhomogeneous $F(1,1)$ or $C(n,2)$ structures, 
where vertices in one set of podes look very different from vertices in another set of podes. 
We previously showed similar behavior near the \ER{} curve. However, ERGMs cannot see this spontaneous emergence of inhomogeneity, as the portions of the
Razborov triangle where such emergent inhomogeneity occurs are all ERGM-invisible.
 
This is not to deny the real successes that ERGMs have had. Starting with \cite{CV} and \cite{CD}, definitions
were given of phases in the $(\alpha, \beta)$ plane. The existence of phases and phase transitions was proven
quickly, far faster than for the analogous problems discussed in this paper. Those are important triumphs. 
Fundamentally, the two approaches to studying large graphs  represent different parts of applied mathematics, with different goals.


\bigskip

\noindent{\bf Acknowledgment.} The work of L.S. was partially
supported by the National Science Foundation under grant DMS-2113468.

\end{document}